\documentclass[onefignum,onetabnum]{siamart190516}

\usepackage{lipsum}
\usepackage{amsmath}
\usepackage{amsfonts}
\usepackage{amssymb}
\usepackage{graphicx}
\usepackage{epstopdf}
\usepackage{bbm}
\usepackage{enumitem}
\usepackage{mathtools}

\usepackage{algorithm,algpseudocode}
\algnewcommand{\algorithmicforeach}{\textbf{for each}}
\algdef{SE}[FOR]{ForEach}{EndForEach}[1]
  {\algorithmicforeach\ #1\ \algorithmicdo}
  {\algorithmicend\ \algorithmicforeach}
\usepackage{subcaption}
\usepackage{hyperref}
\usepackage{bm} 
\usepackage{cleveref} 
\usepackage{enumitem,array}
\usepackage{comment}
\usepackage{verbatim}
\usepackage{overpic}
\usepackage{tikz-cd}
\usepackage{mdframed}
\usepackage{tikz}
\usetikzlibrary{matrix}
\usepackage{todonotes}

\usepackage{cite}

\ifpdf
  \DeclareGraphicsExtensions{.eps,.pdf,.png,.jpg}
\else
  \DeclareGraphicsExtensions{.eps}
\fi

\usepackage{booktabs}
\usepackage{tabularx}

\newtheorem{example}{Example}

\definecolor{mylinkcolor}{RGB}{0,0,130}
\hypersetup{colorlinks,allcolors=mylinkcolor,citecolor=mylinkcolor}

\crefname{section}{Section}{Sections}
\crefname{subsection}{Subsection}{Subsections}

\headers{Distributional Koopman Operator}{}

\title{The Distributional Koopman Operator for Random Dynamical Systems}

\author{Maria Oprea\thanks{Center for Applied Mathematics, Cornell University, Ithaca, NY, USA (\email{mao237@cornell.edu}).} \and Alex Townsend\thanks{Department of Mathematics, Cornell University, Ithaca, NY, USA 
  (\email{townsend@cornell.edu}).} \and Yunan Yang\thanks{Department of Mathematics, Cornell University, Ithaca, NY, USA 
  (\email{yunan.yang@cornell.edu}).}} 

\ifpdf
\hypersetup{
  pdftitle={DKO},
  pdfauthor={}
}
\fi

\begin{document}

\maketitle

\begin{abstract}
The Distributional Koopman Operator (DKO) is introduced as a way to perform Koopman analysis on random dynamical systems where only aggregate distribution data is available, thereby eliminating the need for particle tracking or detailed trajectory data. Our DKO generalizes the stochastic Koopman operator (SKO) to allow for observables of probability distributions, using the transfer operator to propagate these probability distributions forward in time. Like the SKO, the DKO is linear with semigroup properties, and we show that the dynamical mode decomposition (DMD) approximation can converge to the DKO in the large data limit. The DKO is particularly useful for random dynamical systems where trajectory information is unavailable. 
\end{abstract}



\tableofcontents

\section{Introduction}
\label{sec:introduction}

The Koopman operator framework has emerged as a powerful tool for analyzing the dynamics of nonlinear systems by transforming nonlinear evolution into a linear but infinite-dimensional setting in the space of observable functions. Introduced by Koopman in 1931~\cite{koopman1931hamiltonian}, the classical Koopman operator is extensively studied for deterministic dynamical systems, capturing the evolution of scalar observables along trajectories of a system's state~\cite{original_stoch_koopman, budivsic2012applied}. In many practical applications, however, systems are subject to uncertainty or noise due to modeling errors, environmental variability, or inherent randomness, necessitating a stochastic description. Random dynamical systems incorporate such randomness explicitly, modeling state evolution as a stochastic process rather than a deterministic map or flow.

While Koopman theory was extensively developed for deterministic systems, it has recently been extended to random dynamical systems (RDSs) with the Stochastic Koopman Operator (SKO)~\cite{stochastic_koopman} by evolving observables via conditional expectations. In fact, there are several papers in the literature extending the traditional Koopman framework. For example, Zhang et al.~\cite{zhang2022koopman} develops a Koopman-based framework for rare event simulation in random differential equations, focusing on finding rare events that dominate the dynamics. In a similar spirit, Sinha et al.~\cite{sinha2020robust} propose a robust computational techniques for approximating the Koopman operator in the context of random dynamical systems. Furthermore, the spectral properties of the Koopman operator for random systems have been investigated~\cite{vcrnjaric2020koopman}, and additional applications of Koopman operator methods in estimation and control are presented in Otto and Rowley~\cite{otto2021koopman}.

In this work, we continue this line of research by introducing the Distributional Koopman Operator (DKO), which acts on observables of probability distributions. The key idea is to view RDSs as a deterministic evolution in the space of probability distributions via the transfer operator~\cite{lasota1994entropy, klus2016numerical}. DKO uses observables on probability distributions instead of functions of the state space and evolves them using the deterministic flow governed by the transfer operator. The DKO is an operator $\mathcal{D}_t$ such that
\[
\left[\mathcal{D}_t h\right]\!(\pi) = h\big(T_t(\pi)\big),
\]
where $T_t$ is the transfer operator that evolves the probability distribution $\pi$ forward by time $t$ in a way that corresponds to the underlying RDS, and $h$ is an observable. The DKO allows one to apply Koopman analysis to RDSs where data is not collected as a single trajectory or when individual particle tracking is unavailable. Like other Koopman operators, the DKO is linear and forms a semigroup (see~\cref{lem:DKO_Linear}). 

We also present a data-driven dynamical mode decomposition (DMD) algorithm tailored to the DKO framework (see~\cref{algo:DKO}) to construct a finite-dimensional matrix approximation using distributional snapshot data. Unlike the DMD algorithm for the SKO, the DMD algorithm for the DKO does not need a single trajectory. This allows us to apply the DKO framework to dynamical systems where particle tracking is not feasible (see, for example,~\cref{subsec:Dust}). Moreover, our DMD approximation of DKO converges to the best matrix approximation in the large data limit  (see~\cref{sec:DKOHilbert,sec:Convergence}). 

Our DKO framework generalizes SKO in the sense that when the input distributions are Dirac measures, and the DKO observables are linear, i.e., in $H_1$ (see~\cref{eq:linear}), the DKO reduces to SKO (see~\cref{sec:dko_sko_connection}). When single trajectory information is collected from an RDS and linear observables are chosen, the proposed framework constructs the same operator as SKO. However, the DKO framework continues to make sense without single trajectory information and can take nonlinear observables into consideration, such as variance (see~\cref{sec:DKOobservables,subsec:variance}). Our DKO can be applied to any Markovian dynamics, where the future evolution of an RDS is fully determined by its current state alone, independent of its historical states. While the quality of the finite-dimensional approximation of the DKO depends on the choice of observables and the type of snapshot data collected, the matrix approximation can always be constructed. 

The paper is structured as follows. In~\cref{sec:background}, we review background material on RDSs,  the SKO, and the DMD algorithm for SKO approximation.~\Cref{sec:DKO} introduces the DKO, establishes its theoretical properties, and discusses its relationship with SKO when restricted to linear observables. In~\cref{sec:DKOHilbert}, we develop a finite-dimensional regression framework to approximate the DKO, and we then prove convergence in the infinite-data limit in~\cref{sec:Convergence} utilizing functional analysis over the space of probability distributions. Numerical experiments are presented in~\cref{sec:numerics} to demonstrate the efficacy of the proposed approach using datasets without individual particle tracking information. 

\section{Background}\label{sec:background} 
This section presents some background material on RDSs (see \Cref{sec:random_dynamical_system}), the transfer operator (see~\Cref{sec:TransferOperator}),  SKO (see \Cref{subsec:sko}), and DMD approximation for the SKO (see \Cref{subsec:DMD4SKO}).

\subsection{Random Dynamical Systems}\label{sec:random_dynamical_system}
An RDS is a dynamical system where the evolution of trajectories is random in nature. Let $(\Omega, \mathcal{F}, p)$ be a probability space and consider a semigroup of measurable maps $\theta_t\colon \Omega \to \Omega$ that preserves the measure $p$. For any $\omega\in \Omega$, the shifted element $\theta_t\omega$ is also random, and we assume that $\omega$ and $\theta_t\omega$ are independent.  An RDS is a map 
\[
\Phi\colon \mathbb{R}\times\Omega\times M \to M,
\]
which is measurable with respect to $\mathcal{B}\times\mathcal{F}$ (with $\mathcal{B}$ the Borel $\sigma$-algebra on $\mathbb{R}$). We think of the first argument of $\Phi$ as the time parameter, the second as encapsulating the randomness, and the third as the state. When the RDS is discrete-time,  not continuous, we restrict the first argument of $\Phi$ to an integer, i.e.,  $\Phi\colon \mathbb{N}\times\Omega\times M \to M$. 

We write $\Phi_t$ to denote the evolution map at the fixed time $t$ and $\Phi_n$ for a discrete-time RDS at timestep $n$. In other words, we have
\[
\Phi_t(\omega,x) = \Phi(t,\omega,x), \quad \text{ or } \quad \Phi_n(\omega,x) = \Phi(n,\omega,x),
\]
depending on whether the RDS is continuous-time or discrete-time. 
The map $\Phi$ satisfies the following cocycle properties:
\[
\Phi_0(\omega,\cdot) = \operatorname{Id}, \quad \Phi_{t+s}(\omega,\cdot) = \Phi_t\bigl(\theta_s\omega,\cdot\bigr)\circ \Phi_s(\omega,\cdot),
\]
where $\operatorname{Id}$ is the identity operator.  An RDS admits two equivalent interpretations:
\begin{itemize}[noitemsep,leftmargin=*]
    \item As a deterministic system on the product space $\Omega\times M$ via the formulation: 
    \[
    \Psi_t(\omega,x) = \bigl(\theta_t\omega,\Phi_t(\omega,x)\bigr).
    \]
    \item As a random family of dynamical systems on $M$, where each $\omega\in\Omega$ determines an individual evolution via $\Phi_t(\omega,\cdot)$.
\end{itemize}
\smallskip

We give two examples of RDSs below.
\begin{example}[Random Rotations on the Circle]
Let $M=\mathbb{S}^1$ and consider a discrete-time RDS for which, at each time step, a point is rotated clockwise by angle $\nu$ and then rotated by an angle $\omega$ drawn uniformly at random from $[-\tfrac{1}{2},\tfrac{1}{2}]$. One can interpret this as dynamics on the unit circle where the state is an angle defined modulo $2\pi$. Define $\Omega$ as the set of infinite sequences 
\[
\left\{(\omega_k)_{k\geq1}\mid \omega_k\in[-\tfrac{1}{2},\tfrac{1}{2}]\right\},
\]
equipped with the Bernoulli measure $p$, where the components $(\omega_k)$ are identical and independently distributed (i.i.d.), following the  uniform distribution over $[-\tfrac{1}{2},\tfrac{1}{2}]$. Let $\omega = (\omega_0,\omega_1,\dots)$ be an element of $\Omega$ with the shift $\theta_n$ for $n\in\mathbb{N}$ defined by 
\[
\theta_n(\omega_0,\omega_1,\dots) = (\omega_n,\omega_{n+1},\dots).
\]
It is simple to check that $\theta_n$ preserves the measure $p$ for any $n$~\cite[Prop.~4.4.4]{dynamical_systems_book}.  The corresponding RDS map is given by a recurrence
\[
\Phi_{n+1}(\omega,x) = \left(\Phi_n\bigl(\theta_1 \omega,x\bigr) + \nu +  \omega_n\right)\!\!\!\!\mod\! 2\pi.
\]
Thus, at each step, the state undergoes a rotation by a random angle drawn uniformly from $[-\tfrac{1}{2},\tfrac{1}{2}]$ in addition to the drift $\nu$.
\end{example}

\begin{example}[Stochastic Differential Equations]
The solution to a stochastic differential equation (SDE) naturally defines an RDS~\cite{rds}. Let 
\[
\Omega = \Bigl\{\omega\in \mathcal{C}(\mathbb{R}) : \omega(0)=0\Bigr\}
\]
be the space of continuous real-valued functions starting at the origin, endowed with its Borel $\sigma$-algebra, and let $p$ be the Wiener measure, which is invariant under the shifts~\cite[Thm.~8.4]{sde_book}:
\[
\theta_t\omega(s) = \omega(t+s)-\omega(s).
\]
For the SDE given by 
\[
dX_t = a(X_t)\,dt + b(X_t)\,d\omega_t,
\]
with Lipschitz continuous functions $a$ and $b$ and an initial condition $X_0 = x$, the solution can be expressed in the integral form given by
\[
\Phi_t(\omega,x) = x + \int_0^t b\bigl(\Phi_s(\omega,x)\bigr)\,d\omega_s + \int_0^t a\bigl(\Phi_s(\omega,x)\bigr)\,ds.
\]
\end{example}

\subsection{Transfer Operator}\label{sec:TransferOperator}
For a fixed initial condition $x\in M$ sampled from the probability distribution $\pi_0$, the state at time $t$ is a random variable following the distribution $\pi_t\in \mathcal{P}(M)$, where $\mathcal{P}(M)$ denotes the set of all probability distributions on $M$. A fundamental tool for analyzing an RDS is the \emph{transfer operator}, denoted by $T_t$, which characterizes the evolution of an initial probability distribution $\pi_0\in\mathcal{P}(M)$ under a given dynamical system. Specifically, if $X_0\sim \pi_0$, then the distribution of the evolved state is given by $X_t\sim T_t(\pi_0)$. Before defining $T_t$,  we first explain the notion of the pushforward measure~\cite{topics_in_ot}.

\begin{definition}[Pushforward Measure]
Let $(X,\mathcal{B}_X)$ and $(Y, \mathcal{B}_Y)$ be measurable spaces with $\mathcal{B}_X$ and $\mathcal{B}_Y$ being $\sigma$-algebras on $X$ and $Y$, respectively.  Let $f: X \to Y$ be a measurable function and $\mu$ a measure on $X$. The pushforward measure $f {\#}\mu$ of $\mu$ under $f$ is a measure on $Y$ defined by
\[
f{\#} \mu(B) = \mu(f^{-1}(B))\,,
\]
for all $ B \in \mathcal{B}_Y$.
\end{definition}

\begin{definition}[Transfer Operator for an RDS]\label{def:transfer}
Let $\Phi\colon\mathbb{R}\times \Omega\times M\to M$ be an RDS. The corresponding \emph{transfer operator} $T_t\colon\mathcal{P}(M)\to\mathcal{P}(M)$ is defined via
\[
\int_M \hat{h}(x)\,d\bigl(T_t\pi\bigr)(x) = \int_\Omega \int_M \hat{h}(x)\,d\Bigl(\Phi_t(\omega,\cdot)\#\pi\Bigr)(x)\,dp(\omega),
\]
for every test function $\hat{h}\in C_c^\infty(M)$, where $C_c^\infty(M)$ denotes the space of infinitely differentiable functions on $M$ with compact support.
\end{definition}
The transfer operator $T_t$ satisfies similar cocycle properties as the map $\Phi$,  i.e.,  $T_{t+s} = T_t\circ T_s$ for $t,s\geq 0$ (see~\cref{lemma:Pt_semi_group}). 

By casting an RDS as a deterministic evolution on $\mathcal{P}(M)$ through the transfer operator $T_t$, we have a framework encompassing both deterministic and random dynamics. This perspective is central to our DKO framework.

\subsection{The Stochastic Koopman Operator}\label{subsec:sko}
The SKO provides a way to analyze the evolution of observables under random dynamics~\cite{stochastic_koopman,original_stoch_koopman}. The SKO observables can be any bounded functions defined over the state space $M$.   The SKO is a deterministic operator because it characterizes the evolution of functions by averaging over the randomness of the dynamical system. 

\begin{definition}[SKO]\label{def:SKO}
Let $t\geq 0$ and $\Phi\colon \mathbb{R}\times \Omega \times M \to M$ be an RDS, where $(M,\mathcal{B}_M)$ is a Polish space. The SKO,  denoted by $\mathcal{S}_t$, is an operator from observables in $\mathcal{L}^\infty(M)$ to $\mathcal{L}^\infty(M)$ such that 
$$
\left[\mathcal{S}_t \hat{h}\right](x) \coloneqq \mathbb{E}_{\omega\sim p}\Bigl[\hat{h}\bigl(\Phi_t(\omega,x)\bigr)\Bigr] = \int_\Omega \hat{h}\bigl(\Phi_t(\omega,x)\bigr)\,dp(\omega),\qquad x\in M,
$$
for every $\hat{h} \in \mathcal{L}^\infty(M)$. Here, $\mathcal{L}^\infty(M)$ is the space of essentially bounded measurable functions on $M$ and $p$ is the probability distribution governing $\Omega$.
\end{definition}

The operator $\mathcal{S}_t$ maps an observable of the state space at time $0$ to the expected value of the observed random outcome of the dynamics at time $t$. Thus,  one can estimate the value of $[\mathcal{S}_t \hat{h}](x)$ by initializing the random dynamics at $x$ hundreds of times,  waiting for time $t$,  and then averaging the values of $\hat{h}$. 

The SKO is closely related to both the backward Kolmogorov equation and the transfer operator associated with the underlying random dynamics. Notably, the SKO operator is defined pointwise for $x\in M$, and the evaluation requires simulating hundreds of random trajectories starting from each $x$. These assumptions are natural in many theoretical settings but can be restrictive in applications where repeated experiments and measurements are impracticable. Our DKO framework avoids making these assumptions. 

For the SKO to be useful, we require the family of operators $\left\{\mathcal{S}_t\right\}_{t\geq 0}$ to be consistent in a certain sense that they satisfy a semigroup property. For the SKO corresponding to an RDS system, the semigroup property is guaranteed by the fact that the transfer operators satisfy a semigroup property (see~\cite[Sec.~2.1]{stochastic_koopman}).

\subsection{Finite-Dimensional Approximation for SKO}\label{subsec:DMD4SKO}
A popular algorithm for computing an approximation to the SKO is DMD~\cite{rowley2009spectral,original_stoch_koopman,williams2015data}. Given bounded measurable observables $\hat{h}_1, \dots, \hat{h}_n:M \to \mathbb{R}$ and a single trajectory from the RDS, i.e., $x_0,x_1,\ldots,x_m$, which are snapshots sampled at fixed time intervals of length $\Delta t$, DMD approximates the action of the SKO on the span of the observables via least squares. The step-by-step procedure is shown in~\cref{algo:SKO}, which returns an $n\times n$ matrix approximation $S_m$ of $\mathcal{S}_{\Delta t}$.  \cref{algo:SKO} converges to a Galerkin approximation of the SKO under the following assumptions: (1) the single trajectory $\{x_j\}_{j = 0}^m$ sample the entire state space as $m\rightarrow \infty$ and (2) we can interchange space and time averages~\cite{thesis_SKO}. 
\begin{algorithm}
\caption{DMD for Stochastic Koopman Operator}\label{algo:SKO}
\begin{algorithmic}[1]
        \State Given single trajectory data $\{x_j\}_{j = 0}^{m}$ and $n$ different SKO observables  $\{\hat{h}_i\}_{i =1}^n$ with $m\geq n$.
        \State Compute $(\Psi_m)_{ij} = \hat{h}_i(x_{j-1})$  and $(\Phi_m)_{ij} = \hat{h}_i(x_{j})$, $1\leq i \leq n$, $1 \leq j \leq  m $.
        \State \Return $S_m = \Phi_m\Psi_m^\dagger$, where ${}^\dagger$ is the matrix pseudoinverse.
\end{algorithmic}
\end{algorithm}

There are other proposed extensions to computing quantities related to the SKO.   For example,  one can find a matrix approximation using a stochastic Taylor expansion in the expression of the infinitesimal generator~\cite{yuanchao2024}. One can leverage similar ideas in the Residual DMD~\cite{colbrook2024rigorous,ResDMD} to compute the variance of the prediction~\cite{colbrook2024beyond}. Finally,  one can orthogonally project future snapshots onto the subspace of past snapshots to deal with both random dynamics and noisy observations~\cite{Takeish2025}. 

\section{The Distributional Koopman Operator}\label{sec:DKO}
In this section, we define the DKO and discuss how to formulate it as an operator between Hilbert spaces so that we can make sense of its spectral properties.   One can view the DKO as extending the SKO to allow for the dynamical system's state to be described by a probability distribution instead of a location in state space.  

\subsection{Definition of the DKO}
The central idea of our DKO is for the operator to act on continuous functions defined on $\mathcal{P}(M)$,  not bounded functions defined on $M$.   For example,  a DKO observable could represent the mean or the variance of a function of the random state. The DKO is written in terms of the transfer operator, analogous to the Koopman operator definition for deterministic dynamical systems. Therefore, the DKO is a deterministic operator because it describes how functions of probability distributions evolve under the (random but Markovian) dynamics.

\begin{definition}[DKO]\label{def:DKO}
Let $t\geq 0$.  Consider an RDS given by the map $\Phi:\mathbb{R} \times \Omega \times M \to M$ with transfer operator $T_t$ (see~\cref{def:transfer}). For any continuous and bounded function $h:\mathcal{P}(M)\rightarrow \mathbb{R}$ and any $\pi\in \mathcal{P}(M)$,  the DKO is given by
 \[
 \left[\mathcal{D}_t h\right]\!(\pi) = h(T_t(\pi))\,.
 \]
\end{definition}
\Cref{def:DKO} tells us that $\left[\mathcal{D}_t h\right]\!(\pi)$ returns the value of the observable $h$ at $T_t(\pi)$. In practice,  one could estimate the value of $\left[\mathcal{D}_t h\right]\!(\pi)$ as follows: Draw samples hundreds of times from $\pi$,  evolve those samples under the random dynamics for time $t$ without necessarily tracking each trajectory, and then evaluate $h$ at the resulting ensemble distribution $T_t(\pi)$. 

For any $t\geq 0$,  the DKO in~\Cref{def:DKO} has the usual properties of a Koopman operator, i.e., linearity and a semigroup property. 
\begin{lemma}\label{lem:DKO_Linear} The DKO satisfies the following properties:
\begin{enumerate}
    \item  Linearity: For any $\alpha,\beta\in\mathbb{R}$ and any continuous DKO observables $h$ and $g$, we have $\mathcal{D}_t (\alpha h + \beta g) = \alpha \mathcal{D}_th + \beta \mathcal{D}_tg$,
    \item Semigroup property: For any $s,t\geq 0$, we have $\mathcal{D}_{t + s} = \mathcal{D}_t \circ \mathcal{D}_s$.
\end{enumerate}
\end{lemma}
\begin{proof}
The linearity property follows from the fact that $\mathcal{D}_t$ is a composition operator.   In particular, we have
    \[
    \left[\mathcal{D}_t (\alpha h + \beta g)\right]\!(\pi) \!= \!(\alpha h + \beta g)(T_t(\pi)) \!= \!\alpha h(T_t (\pi)) + \beta g(T_t(\pi)) \!= \!\alpha \left[\mathcal{D}_t h\right]\! (\pi) +\beta \left[\mathcal{D}_t g\right]\! (\pi) \,.
    \]
    The semigroup property of $\mathcal{D}_{t}$ follows from the semigroup property of the transfer operator $T_t$, i.e.,  for any $\pi \in\mathcal{P}(M)$, we have
    \[ \left[\mathcal{D}_{t + s} h\right]\! (\pi) = h(T_{t + s}(\pi)) = h(T_t(T_s(\pi))) =\left[ \mathcal{D}_th\right] \!(T_s(\pi)) = \left[\mathcal{D}_t\left[\mathcal{D}_s h\right]\right]\!(\pi)\,.\]
\end{proof}

\subsection{DKO Observables}\label{sec:DKOobservables}
For DKO, observables are not ordinary functions on the state space $M$; rather, they are continuous and bounded functions of probability distributions. That is,  the set of DKO observables can be written as 
\[
    H \coloneqq \{h\colon \mathcal{P}(M) \to \mathbb{R} : h \text{ is continuous and bounded}\}\,.
\]
The space of DKO observables can be categorized as a nested sequence of subspaces,  denoted by $H_1 \subset H_2 \subset \cdots \subset H_n \subset \cdots,$ where
\[
H_n \coloneqq \Bigl\{h\in H \;:\; \frac{\delta^k h(\pi)}{\delta \pi^k} = 0 \quad  k \geq n+1,\; \forall\, \pi \in \mathcal{P}(M) \Bigr\}\,,
\]
where $\delta^k h(\pi)/\delta \pi^k$ is the $k$th Fréchet derivative of $h$ with respect to $\pi$.  

In particular,  $H_1$ contains all the DKO observables such that second and higher-order derivatives are zero,  which implies that they are linear observables. Because of this,  we can write
\begin{equation}\label{eq:linear}
H_1 = \operatorname{span}\Biggl\{ h\colon \mathcal{P}(M) \to \mathbb{R} \,  : \, h(\pi) = \int_M \hat{h}(x)\,d\pi(x),\; \hat{h}\in \mathcal{L}^\infty(M) \Biggr\}\,.
\end{equation}
Therefore, even though $\mathcal{P}(M)$ is not a linear space,  the fact that $h(\pi)$ can be expressed as an integral means that for any $h\in H_1$, it holds that
$$
h(\alpha \pi_1 + (1-\alpha) \pi_2) = \alpha\, h(\pi_1) + (1-\alpha)\, h(\pi_2),  
$$ 
for any $\alpha$ between $0$ and $1$ and any $\pi_1,\pi_2\in\mathcal{P}(M)$.  As a result, $H_1$ can be naturally identified with the space of linear observables with respect to $\pi$.

It is quite natural to restrict the action of the DKO to linear observables that belong to $H_1$ because $H_1$ is always an invariant subspace of $\mathcal{D}_t$ as stated in~\cref{lemma:invariance1} below. 
\begin{lemma}\label{lemma:invariance1}
    The subspace $H_1$ of $H$ is always invariant under the action of $\mathcal{D}_t$. 
\end{lemma}
\begin{proof}
We must show that $\mathcal{D}_t h \in H_1$ for any $h\in H_1$. For any $h\in H_1$, we have 
    \begin{align*}
         \left[\mathcal{D}_t h\right]\! (\pi) = h(T_t(\pi))) &= \int_\Omega \int_M  \hat{h}(x) d(\Phi_t(\omega, \cdot)\# \pi) dp(\omega) \\
         &=\int_\Omega \int_M \hat{h}(\Phi_t(\omega, x)) d\pi(x) dp(\omega)  \\
         & = \int_M \Big(\int_\Omega \hat{h}(\Phi_t(\omega, x)) dp(\omega) \Big)\,d\pi(x)  = \int_M \hat{g}(x) d\pi(x)\,,
    \end{align*}
    where $\hat{g}(x) =\int_\Omega \hat{h}(\Phi_t(\omega, x)) dp(\omega) $. Here, we can swap the integral signs because of Fubini's theorem. Thus, $\mathcal{D}_t h \in H_1$ provided $\hat{g} \in \mathcal{L}^\infty(M)$.  Using H\"older's inequality, we find that 
\[
        \|\hat{g}(x)\|_\infty = \Big\|\int_\Omega\hat{h}(\Phi_t(\omega, x)) dp(\omega)  \Big\|_\infty \leq  \|\hat{h}\|_\infty \int_\Omega dp(\omega)= \|\hat{h}\|_\infty < \infty\,,
\]
where $\|\cdot\|_\infty$ is the supremum norm over $M$. 
\end{proof}
Therefore, if we are only interested in linear observables in the DKO framework,  we can safely work with the space $H_1$ as it is an invariant subspace of $H$ under the action of $\mathcal{D}_t$. 

However, we can also look at other types of observables. The subspace $H_2$ comprises all the ``quadratic" DKO observables. These are the observables for whose Fréchet derivatives vanish beyond the second order, allowing us to express $H_2$ as
\[
    H_2 = \operatorname{span}\Biggl\{ h(\pi) = \Bigl(\int_M \hat{h}_1(x)\,d\pi(x)\Bigr) \Bigl(\int_M \hat{h}_2(x)\,d\pi(x)\Bigr) \,: \,\hat{h}_1,\hat{h}_2\in \mathcal{L}^\infty(M) \Biggr\}.
\]
It is important to notice that the variance of any $\hat{h}\in \mathcal{L}^\infty(M)$ over a given distribution $\pi$ is given by
$$
\operatorname{Var}_{\hat{h}}(\pi) = \int_M \hat{h}^2(x)\,d\pi(x) - \Biggl(\int_M \hat{h}(x)\,d\pi(x)\Biggr)^2,
$$
and therefore, is a function in $H_2$. Although the variance $\operatorname{Var}_{\hat{h}}(\pi)$ is nonlinear with respect to $\hat{h}$, it is a continuous and bounded observable with respect to $\pi$, which makes it a natural quantity of interest to study under the proposed DKO framework. One big difference between the DKO and the SKO is that our DKO framework can also be applied to predict and analyze properties of random variables' variances, not just their expectations (see~\cref{subsec:variance}).  

In principle,  one can also select DKO observables from $H\setminus H_2$, e.g., higher-order central moment of the random state variable. However, in this paper, we restrict our attention to DKO observables from $H_1$ and $H_2$. 

\subsection{The DKO as a Generalization of the SKO}\label{sec:dko_sko_connection}
The SKO can be viewed as a specific case of the DKO, and this relationship can be made precise by restricting the DKO to its invariant subspace $H_1$. More specifically, there are two distinct approaches to obtaining the SKO from the DKO:
\begin{itemize}[noitemsep,leftmargin=*]
    \item If one restricts the DKO observables to be in $H_1$,  then for any $h\in H_1$ and $x\in M$,  we have 
\[
\left[\mathcal{S}_t \hat{h}\right]\! (x) = \left[\mathcal{D}_t h\right]\! (\delta_x)\,,
\]
where $\hat{h}$ is defined by $\hat{h}(x) := h(\delta_x)$ and $\delta_x$ is the Dirac measure centered at $x$.   Hence,  we can view the SKO as a restricted version of the DKO on single Dirac delta measures over  $M$.

\item Another interesting way to view the DKO as a generalization of the SKO is if one restricts the DKO observables to be in $H_1$, then on the one hand, from~\Cref{def:DKO} and the proof of~\Cref{lemma:invariance1},  we have 
\[
    \left[\mathcal{D}_{t} h\right] \! (\pi) =  \int_M \left( \int_\Omega \hat{h}(\Phi_{t}(\omega, x)) dp(\omega)\right) d\pi(x)
\]
for any $h\in H_1$ and $\hat{h}$ satisfying $h(\pi) = \int_M \hat{h}(x) d\pi(x)$, $\forall \pi \in \mathcal{P}(M)$.  On the other hand,  we have 
\[
    \left[\mathcal{S}_{t} \hat{h}\right]\!(x) = \int_\Omega \hat{h}(\Phi_{t}(\omega, x)) dp(\omega).
\]
Thus, we find that 
\[
 \left[\mathcal{D}_{t} h\right] \! (\pi)  = \int_M \left[\mathcal{S}_{t} \hat{h}\right]\!(x) d\pi (x) = \mathbb{E}_{X\sim \pi}\!\!\left[\left[\mathcal{S}_{t} \hat{h}\right]\!(X)\right].
 \]
Hence, we can view the DKO restricted to $H_1$ observables as an \textit{averaged version} of the SKO,  where the average is done by the given probability distribution $\pi$.   
\end{itemize}
\smallskip

It is useful to remember that the DKO is a generalization of the SKO when linear DKO observables are used,  except that the DKO does not need single trajectories starting from each $x\in M$ or their tracking information.  To highlight this point even further, we note that the eigenvalues and eigenfunctions of $\mathcal{S}_{t}$ are also eigenvalues and eigenfunctions of $\mathcal{D}_{t}$. 

\begin{lemma}\label{lemma:eigen}
    Let $t\geq 0$ and $\lambda$ be an eigenvalue of $\mathcal{S}_{t}$ with eigenfunction $\hat{h}$. Then, $h(\pi) = \int \hat{h}(x) d\pi(x)$ is an eigenfunction of $\mathcal{D}_{t}$ with eigenvalue $\lambda$. 
\end{lemma}
\begin{proof}
Using the fact that $[\mathcal{S}_{\Delta t} \hat{h}](x) = \lambda \hat{h}(x)$, we find that for any $\pi\in\mathcal{P}(M)$, we have
\[
     [\mathcal{D}_{t} h](\pi) = \int [\mathcal{S}_{t}\hat{h}](x) d\pi(x) = \int \lambda \hat{h}(x) d\pi(x) = \lambda h(\pi)\,.
\]
\end{proof}

Due to the relationship between $\mathcal{S}_t$ and $\mathcal{D}_t$, the DKO framework is most useful when a single trajectory from the RDS is not available (as that is required by~\cref{algo:SKO}) or if one wants to work with nonlinear DKO observables, such as the variance (see~\cref{subsec:variance}).

\subsection{Finite-Dimensional Approximation of DKO}\label{sec:FiniteDimDKO}
While the concept of the DKO is a useful analytical notion, in practice, we only have access to information from $n$ observables applied to $m$ probability distribution snapshots. That is, for some $\Delta t>0$ and $n$ observables $h_1,\ldots,h_n\in H$, we are given $m$ probability distribution snapshots from the RDS of the form:
\[
\begin{aligned} 
h_1(\pi_j), \ldots, h_n(\pi_j), \qquad 1\leq j\leq m,\\
h_1(\mu_j), \ldots, h_n(\mu_j), \qquad 1\leq j\leq m,
\end{aligned}
\]
where $\mu_j = T_{\Delta t}(\pi_j)$ for $1\leq j\leq m$ and $T_{\Delta t}$ is the transfer operator associated with an RDS. We assume that one has collected enough probability distribution snapshot data such that $m\geq n$. If one has access to a single long trajectory of probability distribution snapshots $\{\pi_1, \dots, \pi_{m}, \pi_{m+1}\}$, then one can construct the required pairs by defining $\mu_j = \pi_{j+1}$, $1\leq j \leq m$.

In the same way as the DMD for deterministic dynamical systems, we select a matrix $D_m$ that represents $\mathcal{D}_{\Delta t}$ when restricting its input and output to $V_n = \text{span}\{h_1, \dots, h_n\}$. Since $g(\pi) = \sum_{i=1}^n g_i h_i(\pi)$ for an observable $g \in V_n$ with coefficients $\{g_i\}_{i=1}^n$ and $\left[\mathcal{D}_{\Delta t} g\right]\!(\pi) = g\left(T_{\Delta t}(\pi)\right)$ for all $\pi\in\mathcal{P}(M)$, we would like to find a matrix $D_m$ such that
\begin{equation} 
\sum_{i=1}^n g_{i} h_i(\mu_j) \approx \sum_{i=1}^ng_{i} \sum_{k=1}^n (D_m)_{ik} h_k(\pi_j), \qquad 1\leq j\leq m.
\label{eq:DMDDKO}
\end{equation} 
If we define $(\Psi_m)_{ij} = h_i(\pi_j)$ and $(\Phi_m)_{ij} = h_i(\mu_j)$ for $1\leq i\leq n$ and $1\leq j\leq m$, then~\cref{eq:DMDDKO} becomes
\begin{equation}
\mathbf{g}^\top\Phi_m \approx \mathbf{g}^\top D_m\Psi_m.
\label{eq:approx}
\end{equation}
Since there may be no choice of $D_m$ for which the approximation sign in~\cref{eq:approx} is equality, we solve for $D_m$ as the solution to the least-squares problem: 
\begin{equation}
\min_{D_m\in\mathbb{R}^{n\times n}} \| \Phi_m - D_m\Psi_m \|_F^2, 
\label{eq:LeastSquaresProblem}
\end{equation} 
where $\|\cdot\|_F$ is the matrix Frobenius norm. The least-squares problem in~\cref{eq:LeastSquaresProblem} has a unique solution provided that $\Psi_m$ has full row rank~\cite[Sec.~5.5]{golub2013matrix}.  Therefore,  
\[
D_m = \Phi_m\Psi_m^\dagger, \qquad \Psi_m^\dagger = \Psi_m^\top\left(\Psi_m\Psi_m^\top\right)^{-1}
\]
is our DMD approximation to $\mathcal{D}_{\Delta t}$ (see~\cref{algo:DKO}).  Equivalently, $D_m$ solves the normal equations associated with~\cref{eq:LeastSquaresProblem} so it is the solution to  
\begin{equation}
\Phi_m\Psi_m^\top = D_m\Psi_m\Psi_m^\top.
\label{eq:NormalEquationsDKO}
\end{equation} 

\begin{algorithm}
\caption{Dynamic Mode Decomposition for DKO}\label{algo:DKO}
\begin{algorithmic}[1]
        \State Given $m$ probability distributions $\{\pi_j\}_{j = 1}^{m}$ and $\{\mu_j\}_{j = 1}^{m}$ such that $\mu_{j} = T_{\Delta t}(\pi_j)$, and $n$ different DKO observables $\{h_i\}_{i =1}^n$, where $m\geq n$.
        \State Compute $(\Psi_m)_{ij} = h_i(\pi_j)$  and $(\Phi_m)_{ij} = h_i(\mu_j)$, $1\leq i \leq n$, $1 \leq j \leq  m $.
        \State \Return $D_m = \Phi_m\Psi_m^\dagger$, where ${}^\dagger$ is the Moore--Penrose pseudoinverse.
\end{algorithmic}
\end{algorithm}

It should be noted that the choice of observables $\{h_i\}$ affects the quality of the DMD approximation to $\mathcal{D}_{\Delta t}$. Carefully selected observables and distributions can yield more accurate finite-dimensional approximations of the operator. As with standard DMD, the spectrum and modes of $D_m$ provide insight into the dominant dynamics captured by the DKO.

\Cref{algo:DKO} is fully data-driven as it does not require explicit knowledge of the transfer operator or the underlying RDS, only access to observables evaluated at pairs of input-output distributions.

\section{Placing DKO Observables into a Hilbert Structure}\label{sec:DKOHilbert}
To study the convergence properties of the DMD approximation and the spectral properties of the DKO, we must first endow the space of observables with a Hilbert space structure. However, this is non-trivial because the DKO acts on probability distributions over a state space $M$, i.e., on elements of $\mathcal{P}(M)$, which is infinite-dimensional and lacks linear structure. Consequently, classical $\mathcal{L}^2$ techniques do not apply directly. We overcome this by lifting the problem into the probabilistic setting of random measures, which allows us to build an $\mathcal{L}^2$-like structure on the space of DKO observables via sampling.

\subsection{Random Measures as a Path to Hilbert Structure}
Let $\Theta$ be a parameter space, and suppose $(\Theta, \mathcal{F}, \mathbb{P})$ is a probability space where the strong law of large numbers holds~\cite[Thm.~11.4.1]{dudley2018real}. We equip the space $\mathcal{P}(M)$ with the weak-* topology associated with the convergence of integrals of bounded continuous functions. The following definition of a random measure~\cite{random_measure_book} allows us to consider a probability measure over $\mathcal{P}(M)$, informally, a ``distribution over distributions.''
\begin{definition}[Random Measure]\label{def:random_measure}
A random measure is a measurable map given by 
\[
\Lambda: (\Theta,\mathcal{F},\mathbb{P}) \to \mathcal{P}(M),
\]
with $\mathcal{P}(M)$ endowed with the $\sigma$-algebra associated with the weak-* topology. 
\end{definition}
Random measures are widely used in Wasserstein statistics~\cite{statistics_OT_book,random_measure_book,review_statistics_wasserstein} where each data point is regarded as a probability distribution. 

\subsection{Constructing Dense Random Measures}
In our paper, we  work with ``dense" random measures, which satisfy the following conditions: 
\begin{enumerate}[noitemsep]
    \item The range of $\Lambda$ is dense in $\mathcal{P}(M)$, and 
    \item The support of probability distributions $\mathbb{P}$ is the whole space $\Theta$.
\end{enumerate}
We now give concrete examples of dense random measures $\Lambda$ using a construction rooted in set-theoretic bijections between uncountable and countable-dimensional rationals.

Let $\phi: \mathbb{R} \to \mathbb{Q}^\mathbb{N}$ be a bijection (which exists because $\aleph_0^{\aleph_0} = 2^{\aleph_0}$~\cite[p.~38]{jech2006set}), and define its $d$-dimensional extension as
\[
\phi^{(d)}: \mathbb{R}^d \to (\mathbb{Q}^d)^\mathbb{N}, \qquad \phi^{(d)}(x) = (\phi(x_1), \dots, \phi(x_d)).
\]
Also let $\alpha: [0,1] \to \mathbb{R}$ be a bijection, e.g., $\alpha(x) = \tan\left(\frac{\pi}{2}(2x - 1)\right)$, and define $\alpha^{(d)}$ analogously to $\phi^{(d)}$. Composing $\phi^{(d)}$ and $\alpha^{(d)}$ together gives a bijection
\[
f^{(d)}:[0, 1]^d\to (\mathbb{Q}^d)^\mathbb{N}, \qquad f^{(d)} = \phi^{(d)}\circ \alpha^{(d)},
\]
whose output is an infinite sequence of vectors of length $d$ with rational entries.

\begin{example}[Empirical Measures]\label{ex:empirical}
Let $\Theta = [0,1]^d$, $\mathbb{P}$ the uniform measure on $\Theta$, and define
\[
\Lambda(x) = \lim_{ m \to \infty} \frac{1}{m}\sum_{j = 1}^m \delta_{y_j},\qquad y_j = (f^{(d)}(x))_j.
\]
Then $\Lambda(\Theta)$ is the set of all empirical measures with rational support, which is dense in $\mathcal{P}(M)$, and $\mathbb{P}$ is fully supported on $\Theta$. We conclude that $\Lambda$ is a dense random measure. 
\end{example}

\begin{example}[Gaussian Mixtures]
Let $\Theta = [0,1]^d \times [0,1]^{d \times d}$, $\mathbb{P}$ the uniform measure on $\Theta$, and $\theta = (\mu, \Sigma)$. Then, the measure
\[
\Lambda(\theta) = \lim_{m \to \infty} \frac{1}{m}\sum_{i = 1}^m  \mathcal{N}(\hat{\mu}_{i}, \hat{\Sigma}_i), 
\]
where $\hat{\mu}_i$ is the $i$th vector in the rational sequence $f^{(d)}(\mu)$ and $\hat{\Sigma}_i$ is the matrix such that $(\hat{\Sigma}_i)_{jk}$ is the $i$th rational value in the sequence $f^{(d)}(\Sigma_{jk})$. In other words, $\Lambda$ is an infinite Gaussian mixture \cite{rasmussen1999infinite} with rational coefficients, rational means, and covariance matrices with rational entries. This model is dense in $\mathcal{P}(M)$~\cite{statistics_OT_book}. Moreover, the uniform measure $\mathbb{P}$ is supported on all of $\Theta$. We conclude that this Gaussian mixture random measure is dense. 
\end{example}

\subsection{Constructing a Hilbert Space on $H$}
To build a Hilbert space structure on the space $H$ of observables, we leverage the random measure $\Lambda: \Theta \to \mathcal{P}(M)$ and consider the pushforward measure $\mathbf{P} = \Lambda \# \mathbb{P}$ on $\mathcal{P}(M)$. This allows us to define an $\mathcal{L}^2$-like structure on $H$ via Monte Carlo sampling. The requirement of $\Lambda$ being dense in $\mathcal{P}(M)$ ensures that we have a well-defined inner product over $H$, not just its equivalent classes.

Suppose $\{\pi_j\}_{j=1}^m \subset \mathcal{P}(M)$ are i.i.d.~samples drawn from $\mathbf{P}$, and let $h \in H$. Then, we have 
\[
\frac{1}{m} \sum_{j=1}^m |h(\pi_j)|^2 = \int |h(\pi)|^2 \, d\mathbf{P}_m(\pi), \qquad \mathbf{P}_m= \frac{1}{m} \sum_{j=1}^m \delta_{\pi_j},
\]
where $\delta_{\pi_j}$ denotes the Dirac measure centered at $\pi_j$. That is, for any bounded measurable function $f: \mathcal{P}(M) \to \mathbb{R}$, we have
\[
\int f(\pi) \, \delta_{\pi_j}(\pi) d\pi = f(\pi_j).
\]

To understand convergence as $m \to \infty$, let $f: \mathcal{P}(M) \to \mathbb{R}$ be continuous and bounded, and suppose $\theta_j \in \Lambda^{-1}(\pi_j)$. Then, we have
\[
    \frac{1}{m} \sum_{j=1}^m f(\pi_j) 
    = \frac{1}{m} \sum_{j=1}^m f \circ \Lambda(\theta_j) 
    = \int f \circ \Lambda(\theta) \, d\mathbb{P}_m(\theta), \qquad \mathbb{P}_m = \frac{1}{m} \sum_{j=1}^m \delta_{\theta_j}.
\]
By the strong law of large numbers and the fact that $f \circ \Lambda$ is bounded, we have $\mathbb{P}_m\rightarrow\mathbb{P}$ almost surely and hence,
\[
\lim_{m\rightarrow \infty} \frac{1}{m} \sum_{j=1}^m f(\pi_j) = \int f(\pi) \, d\mathbf{P}(\pi).
\]
In particular, for any $h \in H$, we conclude that
\[
\lim_{m\rightarrow \infty} \frac{1}{m} \sum_{j=1}^m |h(\pi_j)|^2 = \int |h(\pi)|^2 \, d\mathbf{P}(\pi).
\]

The derivation above motivates us to define the following inner product on $H$, which can be approximated using finitely many probability distributions:
\begin{equation} \label{eq:inner_product}
    \langle h_1, h_2 \rangle_\mathbf{P} = \int h_1(\pi) h_2(\pi) \, d\mathbf{P}(\pi), \qquad \text{for all } h_1, h_2 \in H.
\end{equation}

\begin{proposition} \label{prop:inner_product}
Let $(\Theta, \mathcal{F}, \mathbb{P})$ be a probability space and $\Lambda: \Theta \to \mathcal{P}(M)$ a dense random measure. Then,~\cref{eq:inner_product} is an inner product on $H$.
\end{proposition}
\begin{proof}
We verify the three properties of an inner product: (1) symmetry, (2) linearity in the first argument, and (3) positive definiteness.  

Symmetry is immediate from the definition. Moreover, linearity in the first argument follows from the fact that integrals are linear operators. For positive definiteness, for any $h \in H$, we have
\[
\langle h, h \rangle_\mathbf{P} = \int |h(\pi)|^2 \, d\mathbf{P}(\pi) \geq 0.
\]

Finally, when $\langle h, h \rangle_\mathbf{P} = 0$, we have $h(\pi) = 0$ for $\mathbf{P}$-almost every $\pi$. Since $\mathbf{P} = \Lambda \# \mathbb{P}$ and the support of $\mathbb{P}$ is the whole of $\Theta$, the support of $\mathbf{P}$ is the closure of $\Lambda(\Theta)$.  Seeking a contradiction, suppose that $h(\pi_0) \neq 0$ for some $\pi_0 \in \Lambda(\Theta)$. Then, by the continuity of $h$, an open set $U$ exists in the closure of $\Lambda(\Theta)$ such that $h(\pi) \neq 0$ for $\pi\in U$. This contradicts the fact that $h(\pi) = 0$ for $\mathbf{P}$-almost every $\pi$ because the support of $\mathbf{P}$ is dense in $\mathcal{P}(M)$.  Finally, since $\Lambda(\Theta)$ is dense in $\mathcal{P}(M)$ and $h$ is continuous, we conclude that $h(\pi) = 0$ for all $\pi \in \mathcal{P}(M)$, i.e., $h = 0$. Therefore, we have proved that $\langle \cdot, \cdot \rangle_\mathbf{P}$ defines an inner product on $H$. 
\end{proof}

It is worth noting that if the range of the random measure $\Lambda$ is not dense in $\mathcal{P}(M)$, then the inner product defined in~\cref{eq:inner_product} is only an inner product in the space of equivalence classes $H/_\sim$, where $h = 0$ implies that $h(\pi) = 0$ $\mathbf{P}$-almost everywhere. Although one can state the convergence results in~\cref{sec:Convergence} in terms of observables in the space $H/_\sim$, such results might not be informative. This is because the DMD approximation is not valid if the observables are evaluated at distributions outside the closure of the support of $\mathbf{P}$. For deterministic dynamical systems, where the space of observables is in $\mathcal{L}^2(\nu)$ with $\nu$ being the invariant measure supported everywhere on the attractor, the behavior outside of the support of $\nu$ is irrelevant for the long-time dynamics. In this case, predictions of the observables in a $\nu$-almost everywhere sense are acceptable. However, in the stochastic setting, there is no restriction on the uncertainty, so applying the DMD matrix approximation on distributions outside the support of $\mathbf{P}$ can still reveal relevant information about the dynamics.

Since $[\mathcal{S}_t \hat{h}](x) = [\mathcal{D}_t h](\delta_x)$, we want the matrix approximations $S_m $ and $D_m$ to be close when applied to the coefficients of $h$ in the $V_n$ basis and evaluated at $\delta_x$, but we cannot be sure that this is the case when $\delta_x$ is not in the support of $\mathbf{P}$. It is important to remember that we assume $\pi_1, \dots, \pi_m$ are sampled i.i.d.~from $\mathbf{P}$. In practice, the data is represented as empirical distributions. Hence, we would like to select a random measure that allows all possible empirical measures as samples, which is exactly the dense random measure presented in~\cref{ex:empirical}.


\subsection{Constructing a Hilbert--Schmidt Norm}
To quantify the difference between DKO and its approximations, we wish to utilize the Hilbert--Schmidt norm for operators mapping the Hilbert space $H$ to itself. To this end, suppose we select the following subspace of DKO observables $V_n = \text{span}\{h_1, \dots, h_n\}$, where the $h_1,\ldots,h_n$ form a linearly independent basis for $V_n$.  We can define a linear operator $\mathcal{E}: V_n\rightarrow H$ by restricting the input space of the DKO to $V_n$. We define the Hilbert--Schmidt norm of $\mathcal{E}$ as 
\begin{equation} 
\|\mathcal{E}\|_{\text{HS}} = \|G^{-1}E\|_{F},
\label{eq:HilbertSchmidt} 
\end{equation} 
where $G_{ij} = \langle h_i,h_j\rangle_{\mathbf{P}}$ and $E_{ij} = \langle \mathcal{E}h_i,\mathcal{E}h_j\rangle_{\mathbf{P}}$. The definition in~\cref{eq:HilbertSchmidt} is consistent with the standard Hilbert--Schmidt norm with orthonormal bases, but the extra inverse Gram matrix $G^{-1}$ accounts for the possibility that  $h_1,\ldots,h_n$ are not orthonormal under the $\langle \cdot,\cdot\rangle_\mathbf{P}$ inner product.
  
\section{Convergence of the DMD Approximation}\label{sec:Convergence}
Now that we have endowed the space of observables with a Hilbert structure, we can consider the convergence of the DMD approximation. Recall that our DMD approximation to $\mathcal{D}_{\Delta t}$ is given by (see~\cref{algo:DKO})
\[
D_m = \Phi_m \Psi_m^\dagger, \qquad (\Phi_m)_{ij} = h_i(\pi_j), \quad (\Psi_m)_{ij} = h_i(\mu_j),\quad 1\leq i \leq n,\quad  1\leq j \leq m\,,
\]
where $h_i(\pi_j)$ and $h_i(\mu_j)$ are evaluated at the collected probability distribution snapshots from the random dynamics with $\mu_j = T_{\Delta t}(\pi_j)$. We show that $D_m$ converges to the best approximation of the Koopman operator $\mathcal{D}_{\Delta t}$ on the finite-dimensional subspace $V_n = {\rm Span}\left\{h_1,\ldots,h_n\right\}$ as $m\rightarrow\infty$, if two conditions hold:
\begin{enumerate}[noitemsep,leftmargin=*]
\item The observables $\left\{h_1,\ldots,h_n\right\}$ are linearly independent basis for $V_n$, and
\item The distributions $\pi_1,\ldots,\pi_m$ are i.i.d.~samples from $\mathbf{P}$.
\end{enumerate} 

We denote the best approximation of the Koopman operator $\mathcal{D}_{\Delta t}$ on the finite-dimensional subspace $V_n = {\rm Span}\left\{h_1,\ldots,h_n\right\}$ by $\mathcal{D}_{\Delta t}^n: V_n\subset H\rightarrow V_n$. This linear finite-rank operator can be represented by an $n\times n$ matrix denoted by
$D_\infty$. The operator $\mathcal{D}_{\Delta t}^n$ (equivalently, the matrix $D_\infty$) minimizes the Hilbert--Schmidt norm of the residual operator given by  $\mathcal{E}:V_n\rightarrow H$ where
\[
\left[\mathcal{E}g\right]\!(\pi) = \left[\mathcal{D}_{\Delta t} g\right]\! (\pi)- \left[\mathcal{D}_{\Delta t}^n g\right]\!(\pi) =  \sum_{i=1}^n g_ih_i(T_{\Delta t}(\pi)) - \sum_{i=1}^n g_i\sum_{k=1}^n (D_\infty)_{ik} h_k(\pi),
\]
for $g(\pi) = \sum_{i=1}^n g_i h_i(\pi)$.  The Hilbert--Schmidt norm of $\mathcal{E}$ is given by $\|\mathcal{E}\|_{\text{HS}} = \| G^{-1} E\|_F$,  where $G_{ij} = \langle h_i,h_j\rangle_{\mathbf{P}}$ and $E_{ij} = \langle \mathcal{E}h_i,\mathcal{E}h_j\rangle_{\mathbf{P}}$, $1\leq i,j \leq n$.

By first-order optimality conditions,  $D_\infty$ is the unique solution to 
\begin{equation} 
Y = D_\infty G, \qquad Y_{ij} = \int h_i(T_{\Delta t}(\pi))h_j(\pi)\,d\mathbf{P}(\pi), \quad G_{ij} = \int h_i(\pi)h_j(\pi)\,d\mathbf{P}(\pi). 
\label{eq:Dinf}
\end{equation} 
Recall that $D_m$ is the unique solution to $\Phi_m\Psi_m^\top = D_m\Psi_m\Psi_m^\top$ (see~\cref{eq:NormalEquationsDKO}),  which we can also write as the solution to the normal equation: $\tfrac{1}{m}\Phi_m\Psi_m^\top = D_m\tfrac{1}{m}\Psi_m\Psi_m^\top$. 

We are ready to show convergence of $D_m$ to $D_\infty$ in the large data limit as $m\rightarrow \infty$.
\begin{theorem}\label{thm:convergence}
Suppose $\left\{h_1,\ldots,h_n\right\}$ is a linearly independent set of DKO observables (see~\cref{sec:DKOobservables}) and the condition number of matrix $G$ defined in~\cref{eq:Dinf} is bounded. 
If for any fixed $m\in \mathbb{N}$, the probability distributions $\pi_1,\ldots,\pi_m$ are i.i.d.~samples from $\mathbf{P}$, then $D_m$ (see~\cref{algo:DKO}) converges to the best matrix approximation, $D_\infty$, to the Koopman operator $\mathcal{D}_{\Delta t}$ on $V_n = {\rm Span}\left\{h_1,\ldots,h_n\right\}$, i.e., 
\[
\lim_{m\rightarrow \infty}\| D_\infty-D_m\|_F = 0,
\]
where the convergence is $\mathbf{P}$-almost surely.
\end{theorem}
\begin{proof}
By~\cref{eq:NormalEquationsDKO} and~\cref{eq:Dinf},  we find that 
 \[
 \begin{aligned}
 \| D_\infty-D_m\|_F & = \| YG^{-1} - \tfrac{1}{m}\Phi_m\Psi_m^\top(\tfrac{1}{m}\Psi_m\Psi_m^\top)^{-1}\|_F\\
 &\leq \underbrace{\| YG^{-1} - Y(\tfrac{1}{m}\Psi_m\Psi_m^\top)^{-1}\|_F}_{I_1} \\& \qquad\qquad+ \underbrace{\| Y(\tfrac{1}{m}\Psi_m\Psi_m^\top)^{-1} - \tfrac{1}{m}\Phi_m\Psi_m^\top(\tfrac{1}{m}\Psi_m\Psi_m^\top)^{-1}\|_F}_{I_2},
 \end{aligned} 
 \]
 where the last inequality comes from the triangle inequality.   
 
To show that $I_1$ goes to $0$ in the large data limit,  first note that 
\begin{equation}\label{eq:I1}
 I_1 \leq \|Y\|_F \| G^{-1} - (\tfrac{1}{m}\Psi_m\Psi_m^\top)^{-1}\|_F.   
\end{equation}
Since $\lim_{m\rightarrow \infty} \tfrac{1}{m}\Psi_m\Psi_m^\top = G$ and $\|G^{-1}\|_F$ is bounded due to the fact that the singular values of $G^{-1}$ are bounded from above, we know that for a sufficiently large $m$,  we have 
\[
\|G^{-1}E_m\|_F \leq \|G^{-1}\|_F\|E_m\|_F < 1, \qquad E_m = G - \tfrac{1}{m}\Psi_m\Psi_m^\top. 
\]
Hence, we can use a Neumann series expansion as follows: 
\[
\begin{aligned}
\| G^{-1} - (\tfrac{1}{m}\Psi_m\Psi_m^\top)^{-1}\|_F &= \| G^{-1} - (G-E_m)^{-1}\|_F \\
&= \| G^{-1} - (I-G^{-1}E_m)^{-1}G^{-1}\|_F \\
&= \left\|G^{-1} - \sum_{s=0}^\infty (G^{-1}E_m)^sG^{-1}\right\|_F \\
&= \left\|\sum_{s=1}^\infty (G^{-1}E_m)^sG^{-1}\right\|_F \\
&\leq \|G^{-1}\|_F^2\|E_m\|_F + \mathcal{O}\left(\|E_m\|_F^2\right).
\end{aligned} 
\]
Thus,  $\| G^{-1} - (\tfrac{1}{m}\Psi_m\Psi_m^\top)^{-1}\|_F\rightarrow 0$ a.s.~as $m\rightarrow \infty$.  Since $\|Y\|_F$ in~\cref{eq:I1} is bounded, we find that $I_1$ goes to $0$ a.s.~as $m\rightarrow \infty$.   

To show that $I_2$ goes to $0$ as $m\rightarrow\infty$,  note that 
 \[
I_2 \leq  \| Y - \tfrac{1}{m}\Phi_m\Psi_m^\top\|_F\|(\tfrac{1}{m}\Psi_m\Psi_m^\top)^{-1}\|_F.
 \]
Due to the triangle inequality, 
\begin{align*}
\|(\tfrac{1}{m}\Psi_m\Psi_m^\top)^{-1}\|_F & \leq \|G^{-1}\|_F + \| G^{-1} - (\tfrac{1}{m}\Psi_m\Psi_m^\top)^{-1}\|_F \\
&\leq   \|G^{-1}\|_F + \|G^{-1}\|_F^2\|E_m\|_F + \mathcal{O}\left(\|E_m\|_F^2\right).
\end{align*}
Therefore, we find that $\|(\tfrac{1}{m}\Psi_m\Psi_m^\top)^{-1}\|_F$ is bounded for all $m$, almost surely. On the other hand,
 \[
\left( \tfrac{1}{m}\Phi_m\Psi_m^\top\right)_{ij}\! =\! \frac{1}{m}\sum_{k=1}^m h_i(\mu_k) h_j(\pi_k) \xrightarrow{\text{a.s.}} \int h_i(T_{\Delta t}(\pi))h_j(\pi)\,d\mathbf{P}(\pi) = Y_{ij},
 \]
for $1\leq i,j\leq n$. Thus,  $\| Y - \tfrac{1}{m}\Phi_m\Psi_m^\top\|_F\rightarrow 0$ a.s.~with $m\rightarrow \infty$. We then find that $I_2\rightarrow 0$ as $m\rightarrow\infty$ almost surely. This finishes the proof. 
\end{proof}
While~\cref{thm:convergence} looks like an asymptotic result,  one could adapt the proof to show that each entry of $D_m$ converges to $D_\infty$ at the Monte Carlo rate of $1/\sqrt{m}$.  However,  more assumptions are needed on the dynamics and the DKO observables if one wants to derive an explicit error bound on $\|D_\infty - D_m\|_F$. 

\section{Numerical Examples}\label{sec:numerics}
 We now numerically demonstrate the application of DKO with several examples.

\subsection{Random Rotations on a Circle}\label{sec:numerics_rotations}
We begin by looking at the RDS introduced in~\cref{sec:random_dynamical_system}, where $M = \mathbb{S}^1$, $\Omega = \{(\omega_k)_{k = 1}^\infty, \omega_k \in \left[-\frac{1}{2}, \frac{1}{2}\right]\}$, $\mathcal{F}$ is the $\sigma$-algebra generated by cylinder sets, and  $p$ is the Bernoulli measure on $\Omega$ where $\{\omega_k\}$ are i.i.d.~samples from the uniform distribution over $\left[-\frac{1}{2}, \frac{1}{2}\right]$.  Finally, we set $\nu = 1/2$.

We consider both the SKO and DKO over one timestep, denoted by $\mathcal{S}_1$ and $\mathcal{D}_1$, respectively.   For this RDS, the eigenvalues and eigenfunctions of $\mathcal{S}_1$ and $\mathcal{D}_1$ restricted to $H_1$ can be calculated analytically.   Moreover,  $\mathcal{S}_1$ and $\mathcal{D}_1$ restricted to $H_1$ share the same eigenvalues,  and their eigenfunctions are related.  This is because if $\hat{g}$ is an eigenfunction of $\mathcal{S}_1$,  then $g(\pi) = \mathbb{E}_{X\sim \pi} [\hat{g}(X)]$ is an eigenfunction of $\mathcal{D}_1$ restricted to $H_1$, with the same eigenvalue.   It can be shown~\cite{stochastic_koopman} that all the eigenfunctions of $\mathcal{S}_1$ take the form 
\[
\hat{g}_k(x) = e^{\mathrm{i}kx}, \quad k \in \mathbb{Z},
\] 
while the corresponding eigenfunctions of $\mathcal{D}_1$ restricted to $H_1$ take the form
\[
g_k(\pi) = \mathbb{E}_{X\sim\pi}[\hat{g}_k(X)],\quad k \in \mathbb{Z}.
\]
The corresponding eigenvalues are $\lambda_k = \frac{\mathrm{i} - \mathrm{i} e^{\mathrm{i} k}}{k}$ for $k\neq 0$ and $\lambda_0 = 1$.

In the following, we numerically approximate these eigenvalues and eigenfunctions using DMD (see~\cref{algo:SKO,algo:DKO}) and compare their accuracy. 

\subsubsection{Eigenvalues of SKO and DKO}
For SKO, we collect a single long trajectory $x_0,x_1,\ldots,x_m\in \mathbb{R}$ from the RDS starting from $x_0 = 0$. We choose $m= 20,\!000$ so that the SKO has access to $N = 20,\!000$ trajectory samples. We select observables $\{\hat{h}_i\}_{i = 1}^{n}$ with $n = 100$, where each $\hat{h}_i(x)$ is an indicator function given by
$$
\hat{h}_i(x) = \chi_{\Big[ \frac{2\pi (i - 1)}{n}, \frac{2\pi i}{n}\Big)} = \begin{cases}
    1, \text{ if } x\in \Big[\frac{2\pi(i - 1)}{n}, \frac{2\pi i}{n}\Big),\\
    0 , \text{ otherwise},
\end{cases}\quad  1\leq i \leq n \,.
$$ 
\Cref{algo:SKO} returns an $n \times n$ matrix approximation to the SKO. 

For DKO, we take $V_n = \text{Span} \{h_1,\ldots,h_n\}$ with $n = 100$, where 
$$
h_i(\pi) = \mathbb{E}_{X \sim \pi}\! \left[   \hat{h}_i(X)\right]\,,\quad 1\leq i \leq n \,.
$$ 
For the measures $\{\pi_j\}_{j=1}^{m}$ needed for~\cref{algo:DKO}, we select conditional distributions on sub-arcs of the uniform distribution over a circle, i.e.,
$$
\pi_j \in \mathcal{P}(\mathbb{S}^1), \qquad \frac{d\pi_j}{dx}(x)=\begin{cases}
     \frac{m}{2\pi}, & x\in \Big[\frac{2\pi(j-1)}{m}, \frac{2\pi j}{m}\Big),\\
     0, & \text{ otherwise},
 \end{cases}\qquad  1\leq j \leq m \,.
$$
We sample $\pi_j$ using $K$ samples to form an empirical distribution. The same applies to each $\mu_j = T_{1}( \pi_j)$ where  $T_{1}$ is the transfer operator for one timestep. Using~\cref{algo:DKO}, we obtain an $n \times n$ matrix approximation of the DKO. For DKO, we use $m = 20$ empirical distributions, each of which is represented by $K = 1000$ samples. Hence, DKO has access to $N = 20\times 1000 = 20,\!000$ trajectory samples. As a result, both DKO and SKO have access to the same number of data points from the RDS (but not the exact same data).
 
We compute the $10$ eigenvalues closest to $\lambda_1,\ldots,\lambda_{10}$ and the eigenfunctions of $S_m$ and $D_m$ associated with $\lambda_1$ and $\lambda_3$ (see~\cref{fig:circle_rotation}).  Unsurprisingly, given the relationship between the SKO and DKO (see~\cref{sec:dko_sko_connection}) and the fact they were given the same amount of data, $S_m$ and $D_m$ return eigenvalues and eigenvectors with very similar accuracy.  
\begin{figure}
    \centering
    \begin{minipage}{.49\textwidth}
    \includegraphics[width=\linewidth]{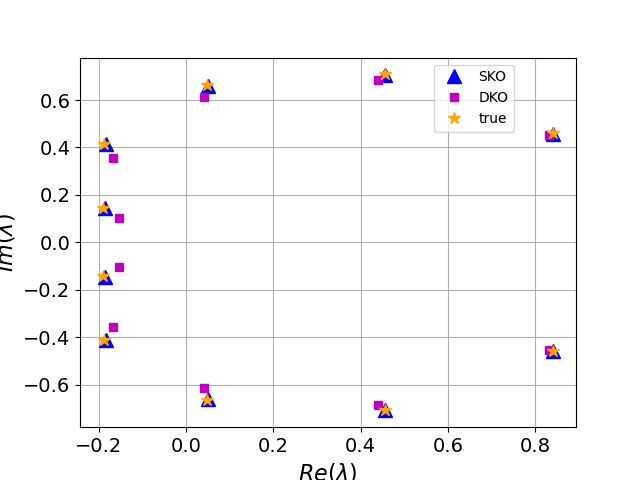}
    \end{minipage} 
        \begin{minipage}{.49\textwidth}
    \includegraphics[width=\linewidth]{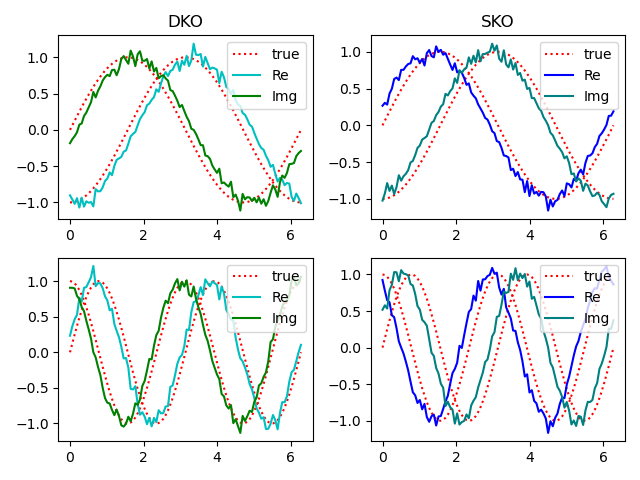}
    \end{minipage}
    \caption{Left: Ten eigenvalues of $S_m$ (blue triangles) and the $D_m$ (red rectangle) closest to $\lambda_1,\ldots,\lambda_{10}$, and compared against the values $\lambda_k = (\mathrm{i} - \mathrm{i}e^{\mathrm{i}k})/k$ (yellow star). Right: The eigenfunction of $S_m$ and $D_m$ associated with $\lambda_1$ (top row) and the eigenfunction  associated with $\lambda_3$ (bottom row), compared against the analytically known functions.}
    \label{fig:circle_rotation}
\end{figure}

\subsubsection{Sensitivity Analysis on Data Amount and Noise Level}
Next, we analyze the performance of SKO (see~\cref{algo:SKO}) and DKO (see~\cref{algo:DKO}) as the amount and the quality of the data changes. 

First, we vary the trajectory data that both algorithms can access. In the prior test, we used $N = 20,\!000$ trajectory samples. Now, following a quadratic progression, we gradually increase $N$ from $1$ to $10,\!000$. Specifically, we construct $S_m$ using a single trajectory with $N = 1, 2^2, \dots, 100^2$. In the meantime, we construct $D_m$ on trajectory data with $m=\sqrt{N}$ distinct empirical distributions, each represented by $K=\sqrt{N}$ samples. We keep the number of observables fixed as before, so $n = 100$. 

To evaluate the performance, we compute the mean-squared error (MSE) between the computed eigenvalues $\{\lambda^{\text{approx}}_k\}$ and the true eigenvalues  $\{\lambda^{\text{true}}_k\}$ of the Koopman operators, which is given by
$$
 \frac{1}{n}\sum_{k = 1}^n \left|\lambda^{\text{true}}_k - \lambda^{\text{approx}}_k\right|^2\,.
$$ 
As we often find, under the same amount of data, both $S_m$ and $D_m$ return eigenvalue approximations that are comparable (see~\cref{fig:errors_circle} (left)).
 
We also test the robustness of~\cref{algo:SKO,algo:DKO} when the trajectory data is noisy. At each step, we assume that the rotation angle is polluted by additive noise following a normal distribution $\mathcal{N}(0,\sigma^2)$. The resulting noisy trajectory $x_{n + 1}$ is obtained by computing $$x_{n + 1}^\text{noise} = (x_n + \nu + \omega_n + \epsilon_n) \ \text{mod}\ 2\pi, \qquad n\geq 0,
 $$
 where $\epsilon_n \sim \mathcal{N}(0, \sigma^2)$, with the standard deviation $\sigma$ denoting the noise level. We vary the noise level $\sigma$ to examine the sensitivity of the eigenvalues of $S_m$ and $D_m$. We compute the MSE between the true eigenvalues and the approximations obtained by~\cref{algo:SKO,algo:DKO}  for $0\leq \sigma \leq 1$ (see~\cref{fig:errors_circle} (right)). We find that the SKO and DKO approaches have comparable robustness to noisy data. 

\begin{figure}
    \centering
    \begin{minipage}{.49\textwidth} 
    \includegraphics[width=\linewidth]{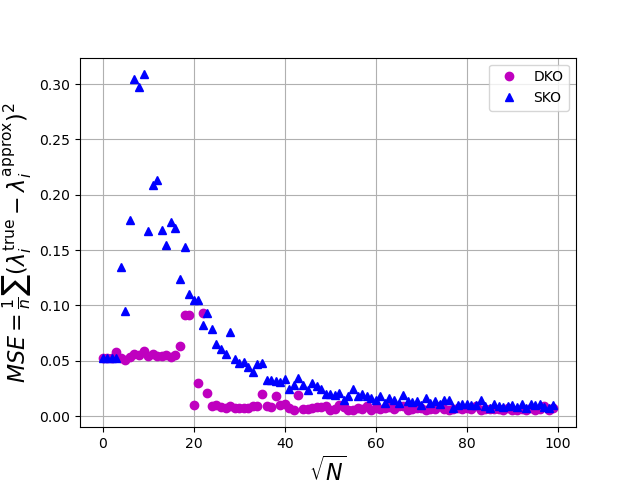}
    \end{minipage} 
        \begin{minipage}{.49\textwidth} 
    \includegraphics[width=\linewidth]{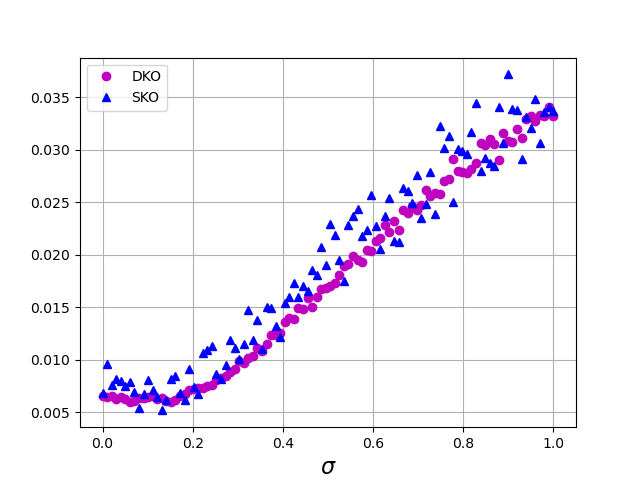}
    \end{minipage}
    \caption{The MSE error between the true and the computed eigenvalues of $S_m$ and $D_m$ as the number of data points increases (left) and as the noise level in the trajectory data increases (right). }
    \label{fig:errors_circle}
\end{figure}

\subsection{Stochastic Differential Equation}\label{sec:sde}
In this subsection, we consider two SDEs with time-independent coefficients (see~\cref{sec:random_dynamical_system}).
In both SDEs, $a(x)$ represents a drift, and $b(x)$ is a diffusion coefficient. We test both a state-independent noise scenario (see~\cref{subsec:noise_independent}), i.e., $b(x)$ is a constant, and a state-dependent noise example (see~\cref{subsec:noise_dependent}).

We select SKO observables as $9$ Gaussian functions centered at equally-spaced points on $[-2,2]$, i.e., 
$$
 \hat{h}_i(x) = e^{-( x - c_i)^2/2}\,,\qquad h_i(\pi) = \int \hat{h}(x) d\pi(x)\,,\quad  1\leq i\leq 9\,,
$$
 where $\{c_i\}_{i=1}^9 = \{-2,-1.5,\ldots, 1.5, 2\}$. We select the DKO observables as the $9$ corresponding functions on probability distributions, denoted by $h_1,\ldots,h_9$. 

In the state-independent and dependent scenarios for the diffusion coefficient, we approximate $\mathcal{D}_{\Delta t}$ with $\Delta t=  0.1$ by constructing the matrix $D_m$. We first draw $K$ samples $\{x^k_0\}_{k = 1}^{K} $ following the standard normal distribution as our initial distribution for the SDE. The ensemble of points serves as initial conditions of $K$ i.i.d.~trajectories, and we obtain points $\{x^k_{j \Delta t}\}_{j = 1}^{m}$ up to a total time $T$ with time interval $\Delta t$ for each $k=1,\ldots, K$. Hence, we obtain $m = T/\Delta t$ empirical measures, each represented with $K$ samples:
\begin{equation}\label{eq:empirical_measure}
    \pi_j =  \frac{1}{K}\sum_{k = 1}^K\delta_{x^k_{(j - 1)\Delta t}}\,, \quad \mu_j = \frac{1}{K} \sum_{k = 1}^K \delta_{x^k_{j\Delta t}}\,,\quad  1\leq j \leq m\,.
\end{equation}

To evaluate the accuracy of the future-time prediction, we compute the reference predicted values for each observable of choice as follows.
We  first approximate the functions 
\begin{equation}\label{eq:pred}
\left[\mathcal{D}_{t}h_i\right]\!(\delta_x) = \mathbb{E}[\hat{h}_i(X_t)|X_0 = x], \, \quad 1\leq i \leq n\,,
\end{equation} 
using Monte Carlo integration for each equidistant point $x$ between $[-2, 2]$. We consider $101$ equidistant points between $[-2, 2]$ and various time points $t = \ell\Delta t$, where $ \ell = 1,\ldots, \lfloor \frac{T_\text{pred}}{\Delta t} \rfloor$ with $T_\text{pred}$ being the maximum prediction time. Specifically, for each equidistant point $x$ between $[-2, 2]$ serving as the initial condition, we compute $N_\text{sample} = 100$  i.i.d.~trajectories following the SDE over the time interval $[0, T_\text{pred}]$ using the Euler--Maruyama method. These trajectories are used to approximate the expectation in~\cref{eq:pred}. Note that the procedure above is not feasible in most realistic situations, and we do it here so that we have true value to evaluate the quality of the predictions produced by DKO.

Once we obtain the matrix approximation $D_m$ for the DKO, we can compute the approximations to $[\mathcal{D}_{\Delta t} {h}_i](\pi)$ for a fixed distribution $\pi$ as follows:
\begin{equation*}
    [\mathcal{D}_{\Delta t}h_i](\pi) \approx \sum_{ k = 1}^n (D_m)_{ik} h_k(\pi).
\end{equation*}
The matrix $D_m^\ell = \underbrace{D_m D_m \ldots D_m}_{\text{$\ell$ times}}$ can be used to do prediction after $\ell\Delta t$ time.

We denote the predicted value of $[\mathcal{D}_{\Delta t}h_i](\delta_x)$ produced by~\cref{algo:DKO} at time $t$ and spatial location $x$ as $f_i(t,x)$.
We compare the numerical prediction using the DKO framework with the ground truth prediction value obtained from direct simulation by computing the Mean-Square Error (MSE):
\begin{align*}
\text{MSE}(t) = \frac{1}{n}\sum_{i = 1}^n \int_{-2}^2 |[\mathcal{D}_t{h}_i](\delta_x) - f_i(t,x) |^2 dx\,.
\end{align*}
We approximate the MSE above based on the trapezoidal rule evaluated at the grid points specified above.

\subsubsection{State-independent Noise}\label{subsec:noise_independent}
We first consider a simple state-independent SDE with drift and diffusion coefficients
$$
a(x) = -x, \qquad b(x) = \sqrt{2}.
$$ 
We consider training time $T = 6$, prediction time $T_\text{pred} = 10$, $K = 100$ and $m = T/\Delta t = 60$.
After running the above procedure, we obtain the results shown in~\cref{fig:SDE} (left). Since each run of~\cref{algo:DKO} produces slightly different datasets due to randomness of the dynamics, we repeat the experiment $100$ times, and compute the mean and standard deviation of MSE. The error bars shown in~\cref{fig:SDE} (left) represent the sample standard deviation divided by $10 = \sqrt{100}$, which is the effective standard deviation of the mean value.

\begin{figure}
\centering
\begin{minipage}{.49\textwidth}
\includegraphics[width=\linewidth]{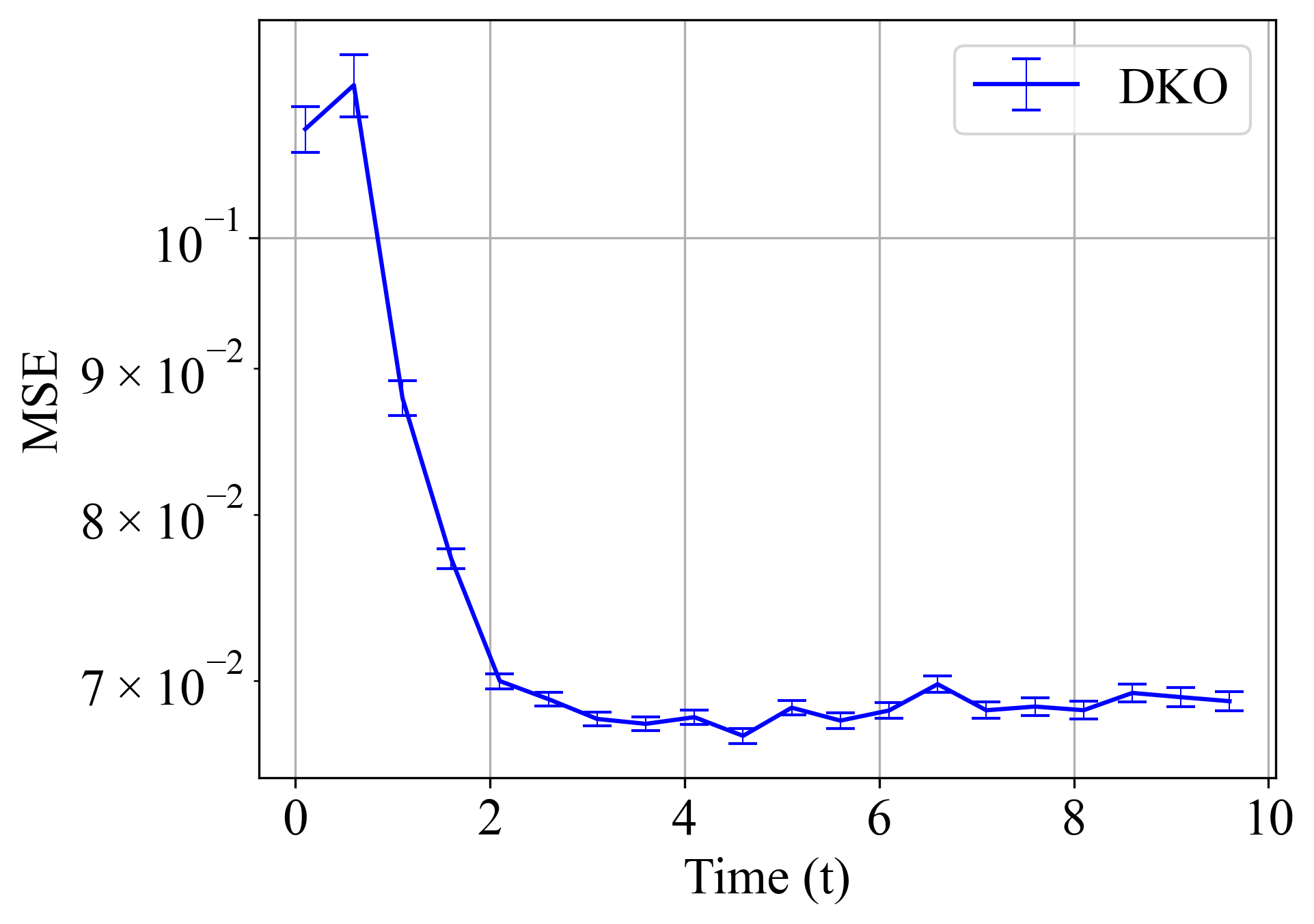}
\end{minipage}
\begin{minipage}{.49\textwidth}
\includegraphics[width=\linewidth]{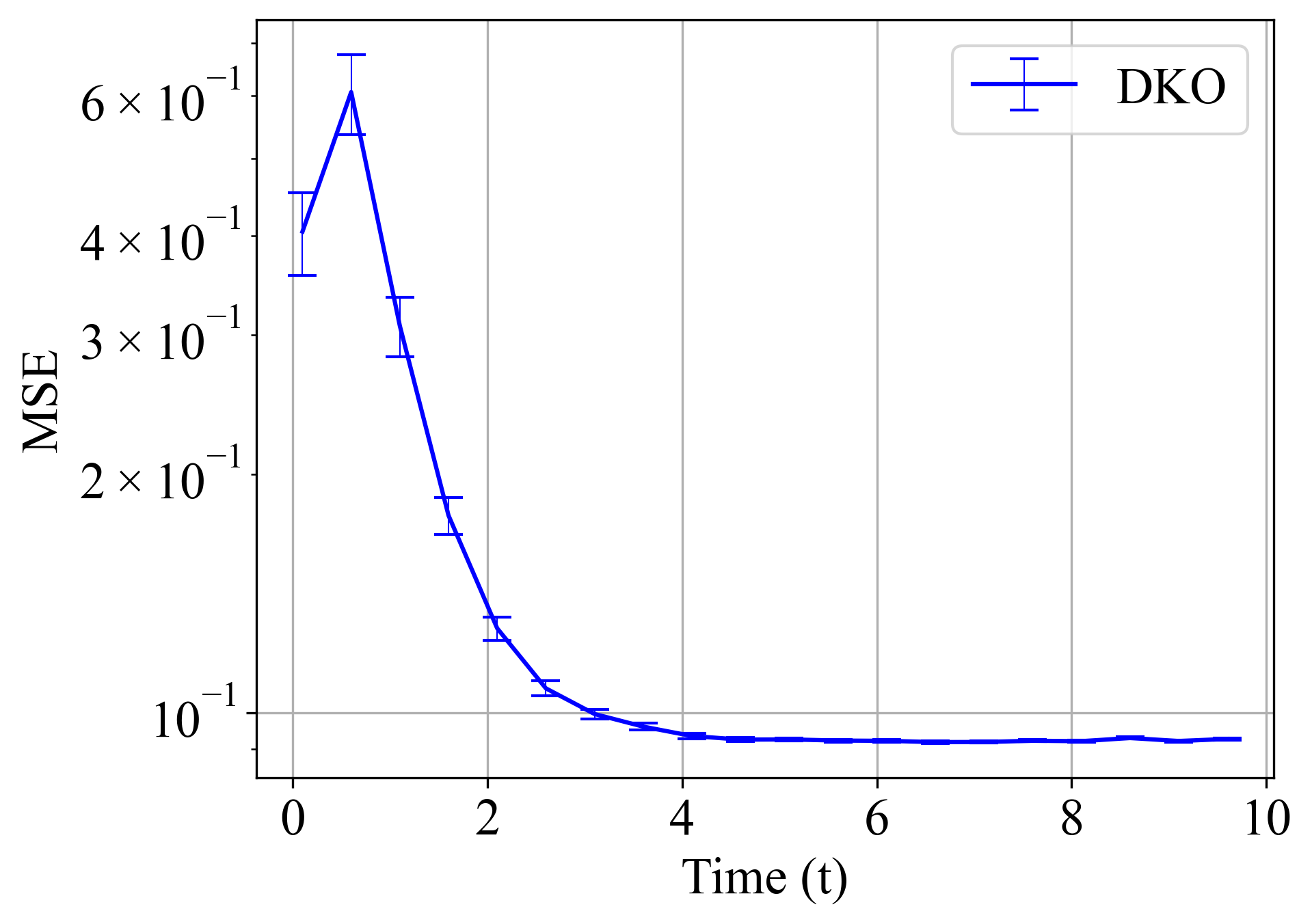}
\end{minipage}
\caption{DKO prediction for state-independent noise (left) and state-dependent noise (right) when the SDE drives the randomized dynamical system. }\label{fig:SDE}
\end{figure}

\subsubsection{State-dependent Noise}\label{subsec:noise_dependent}
Next, we consider the following SDE with a state-dependent diffusion coefficient given by 
$$
a(x) =  -\sin( x) \,, \qquad b(x) = \exp\left( - \frac{1}{2} (x-1)^2 \right)\,.
$$ 
We again consider training time $T = 6$, prediction time $T_\text{pred} = 10$, $K = 100$ and $m = T/\Delta t = 60$. As in~\cref{subsec:noise_independent}, we repeat the experiments $100$ times and plot the mean and standard deviation of the MSE in~\cref{fig:SDE} (right). As in~\cref{subsec:noise_independent}, the mean is normalized by $10 = \sqrt{100}$. 


\subsubsection{Variance Computations}\label{subsec:variance}
One of the benefits of the DKO framework over the SKO framework is that the DKO enables us to predict the variance of an observable of the random state along the trajectory of the SDE.   As before, let $\hat{h}_i(x) = e^{-(x - c_i)^2/2}$ where $\{c_i\}_{i = 1}^9 = \{-2, -1.5, \dots, 1.5, 2\}$ and construct the DKO observables $\{p_{ij}\}$ and $\{q_{ij}\}$ given by
\begin{align*}
p_{ij}(\pi) &=\int \hat{h}_i(x)\hat{h}_j(x)\,d\pi(x),\\
q_{ij}(\pi) &= \left(\int \hat{h}_i(x)\,d\pi(x)\right)\!\!\left(\int \hat{h}_j(x)\,d\pi(x) \right)\,,
\end{align*}
for $1\leq i <j \leq 9$. Then we can consider the DKO acting on the following finite-dimensional linear space: 
\begin{equation}
V = \mathrm{span}\!\left\{\alpha_{ij}p_{ij}+\beta_{ij}q_{ij}: \alpha_{ij},\beta_{ij}\in\mathbb{R}, 1 \leq i < j \leq 9  \right\},\label{eq:V}
\end{equation} 
where all the $p_{ij}$'s and $q_{ij}$'s, in total $90$ observables, serve as the basis function for the space $V$. 
The goal of this test is to approximate the variance of all possible linear combination of $\{\hat{h}_i\}_{i=1}^9$ along the trajectory.

Computationally, we repeat the experiments described at the beginning of~\cref{sec:sde} to compute the best approximation to $\mathcal{D}_{\Delta t}$ when restricted to the subspace $V$ given in~\cref{eq:V}. Our training data consists of $K = 100$ i.i.d.~points sampled from a standard normal distribution. Their corresponding trajectories are denoted by $\{x_{j\Delta t}^k\}_{j = 1}^{m}$, for each $ k = 1,\ldots, K$, for a total time of $T = m \Delta t $ with the timestep $\Delta t = 0.1$ and $m = 60$.  We build empirical measures  $\pi_j$ and $\mu_j$ for $1 \leq j\leq 60$ using~\cref{eq:empirical_measure}. Following~\Cref{algo:DKO}, one can construct a $90\times 90$ matrix $D_m$ and apply it to a vector of coefficients $e\in \mathbb{R}^{90}$, where 
 its $k$th element $e_k$ corresponds to the $k$th basis vector. Here, the $k$th basis vector is $p_{ij}$ such that $9i + j = k$ if $k \leq 45 $, and is $q_{ij}$  such that $9i + j + 45 = k$ if $k > 45$.

Let $g_i(\pi) = {\rm Var}_{X \sim \pi} [\hat{h}_i(X)]$ for $1 \leq i\leq 9$. We assess the quality of our approximation by estimating the function $[\mathcal{D}_{t}g_i](\delta_x)$ where $x\in [-2, 2]$ and for times $t = \ell\Delta t$ with $\Delta t =0.1$ and $\ell  = 1, \dots, 100$, where
$$ 
[\mathcal{D}_{t} g_i](\delta_x) = {\rm Var}\left[\hat{h}_i(X_t)|X_0 = x\right], \, \quad 1\leq i \leq 9\,.
$$
Hence, the DKO prediction of $[\mathcal{D}_{t}g_i](\delta_x) $, denoted by $f_i^{\text{var}} (t, x)$, amounts to repeatedly multiplying the matrix approximation for $\mathcal{D}_{t}$ with the vector $w_i \in \mathbb{R}^{90}$,  whose entries are
\[
(w_i)_j = \begin{cases}
    1, & \text{ if } j = i, \\
    -1, & \text{ if } j = i + 45,\\
    0, & \text{ otherwise},
\end{cases}
\qquad 1 \leq i \leq 9, \quad 1 \leq j \leq 90.
\]
Denote the prediction for different times $t$ and spatial locations $x$ by $f_i^{\text{var}}(t, x)$. To evaluate the accuracy, we compute the MSE given by
$$\text{MSE}_{\text{var}} (t) = \frac{1}{9}\sum_{i = 1}^9 \int_{-2}^2 \left|\left[\mathcal{D}_t g_i\right]\!(\delta_x) - f_i^{\text{var}} (t, x) \right|^2 dx.$$
\Cref{fig:SDE_var2} shows the mean and standard deviation scaled by $10 = \sqrt{100}$ of $\text{MSE}_\text{var}$ computed over $100$ independent experiments, for both the state-independent noise example (see~\cref{subsec:noise_independent}) and the state-dependent noise case (see~\cref{subsec:noise_dependent}).

\begin{figure}
\centering
\begin{minipage}{.49\textwidth}
\includegraphics[width=\linewidth]{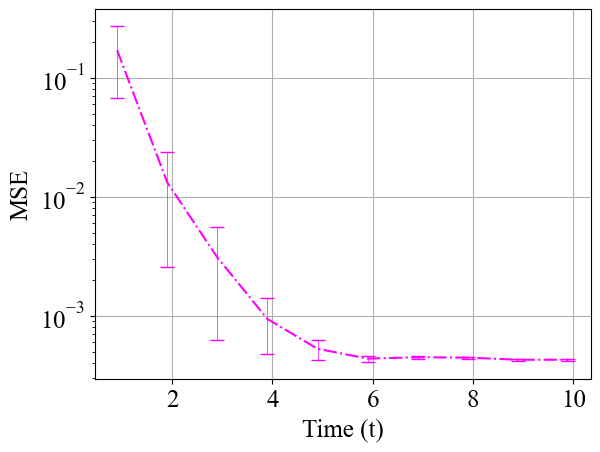}
\end{minipage} 
\begin{minipage}{.49\textwidth}
\includegraphics[width=\linewidth]{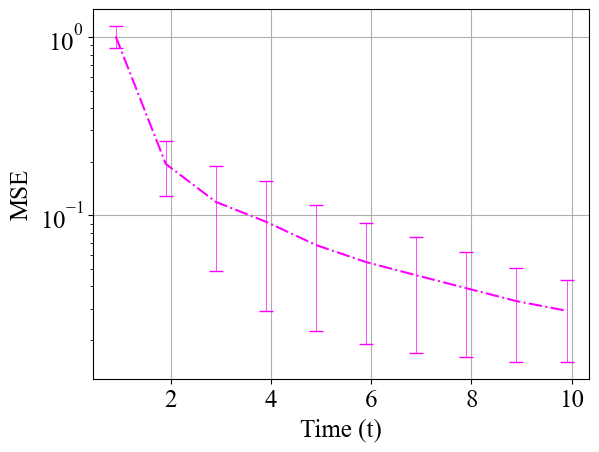}
\end{minipage} 
\caption{Variance prediction for the SDE with state-independent noise  (left) and the SDE with state-dependent noise (right). The test SDEs are shown in~\cref{subsec:noise_independent,subsec:noise_dependent}) and experimental details are discussed in~\cref{subsec:variance}.}\label{fig:SDE_var2}
\end{figure}

\subsection{Dust Plume Data}\label{subsec:Dust}

Dust plume data provide a high-temporal-resolution, event-driven snapshot of dust storm behavior. We use the DustSCAN2022 dataset \cite{plumesdustscan}, which contains hourly dust observations from SEVIRI on Meteosat-8, capturing 8718-time points over the year 2022. Following the methodology outlined in~\cite{plumesdustscan}, we compute the Pink Dust Index (PDI), a normalized measure of dust density derived from the RGB color distance to magenta (see~\cref{fig:dust}, left). Given the lack of trajectory information of individual dust particles, this is an application where SKO techniques cannot be applied. 

To prepare the data for analysis, we coarse-grain the spatial resolution using a structured grid of $50 \times 50$ patches and convert each hourly frame into a vector of average PDI values over these subdomains. This yields a spatiotemporal dataset suitable for DKO. Using the first 28 day's worth of data, corresponding to $24\times 28 = 672$ snapshots, we compute the DKO with the observables being the $50 \times 50$ averaged patches. The DKO eigenvalues (see~\cref{fig:dust}, right) lie within the unit circle, with a few unit-modulus modes corresponding to slowly evolving structures in the dust field.

We assess the forecasting capability of the DKO model by predicting the dust field up to five hours beyond the training window. \Cref{fig:dust_pred} shows snapshots of the actual and predicted fields over the first five hours into the second month (where data exists, but we did not use it to compute the DKO), revealing that the model captures the dominant structures and trajectories of dust plumes during this time. 

\begin{figure}
    \centering
    \begin{minipage}{.55\textwidth}\includegraphics[width=\linewidth]{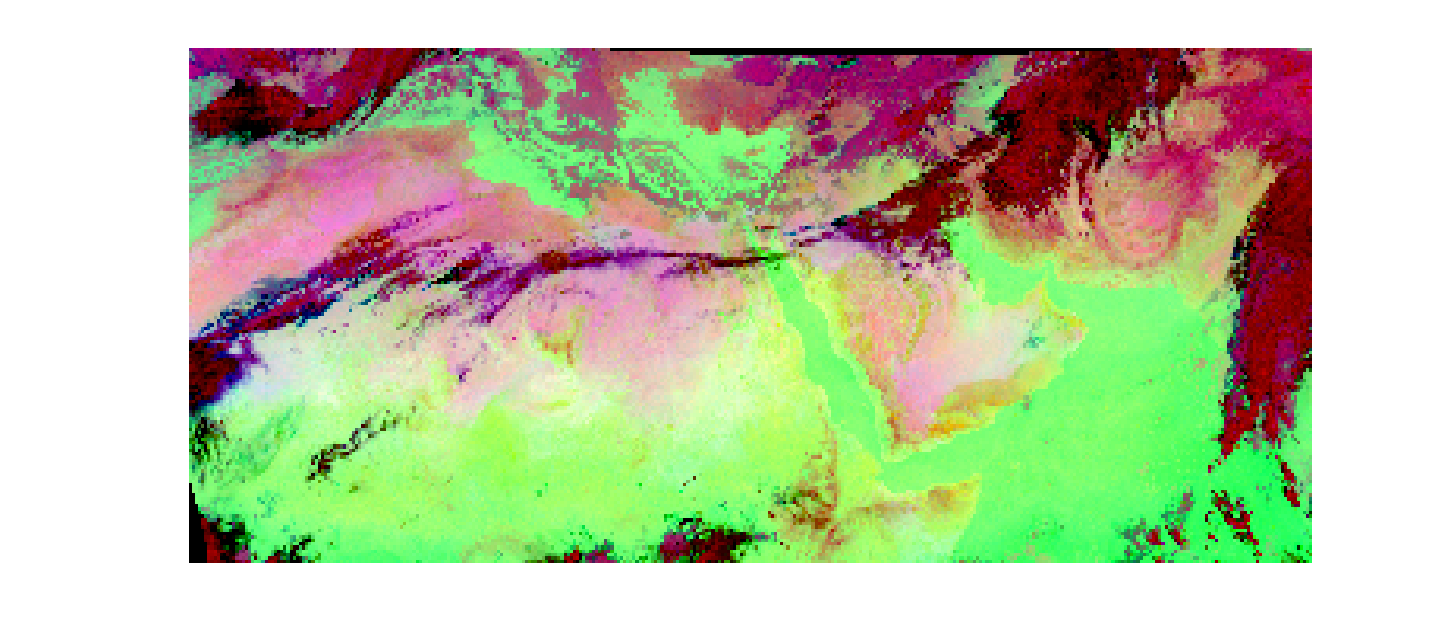}
    \end{minipage}
    \begin{minipage}{.44\textwidth}
    \includegraphics[width=\linewidth]{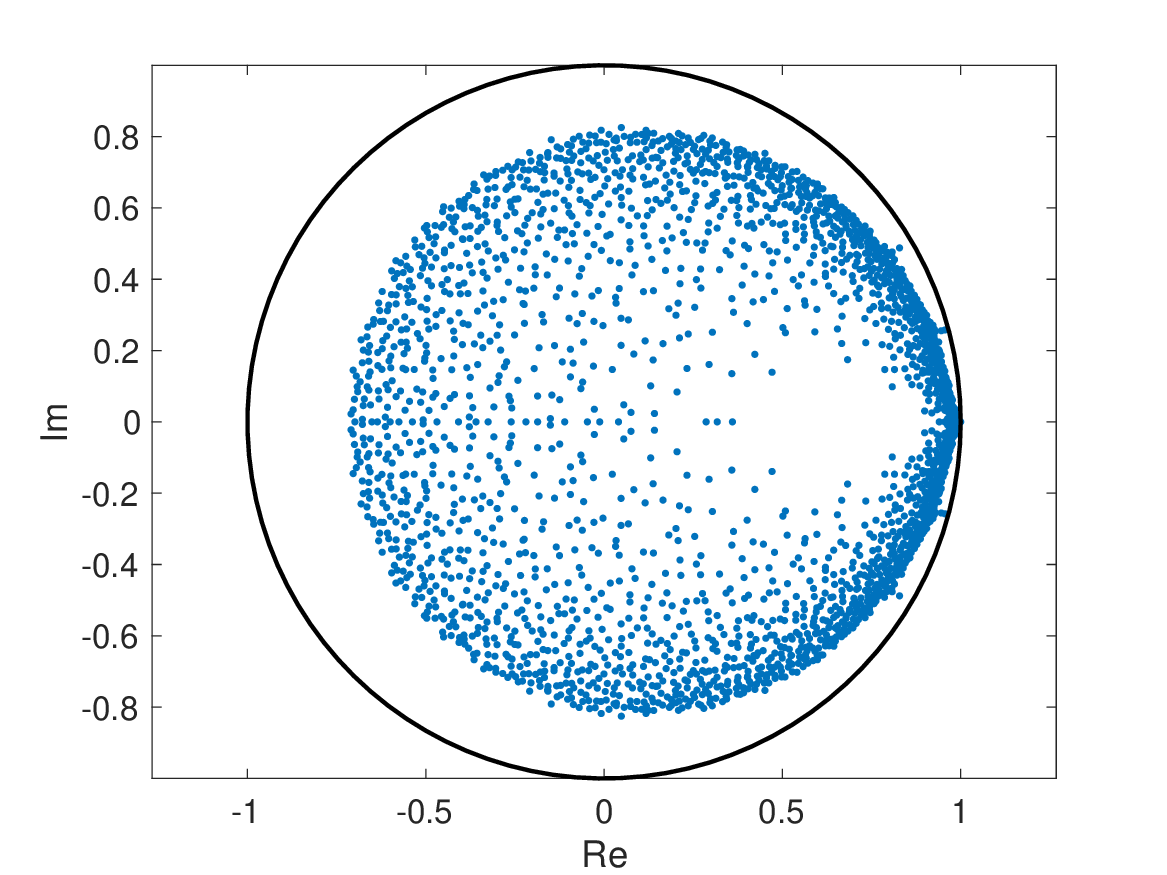}
    \end{minipage}
    \caption{Left: One snapshot of the DustSCAN2022 data, showing dust clouds. Right: DKO eigenvalues (blue dots) along with the unit circle (black line).}
    \label{fig:dust}
\end{figure}

\begin{figure}
    \centering
    \begin{minipage}{.195\textwidth}
        \begin{overpic}[width=\linewidth]{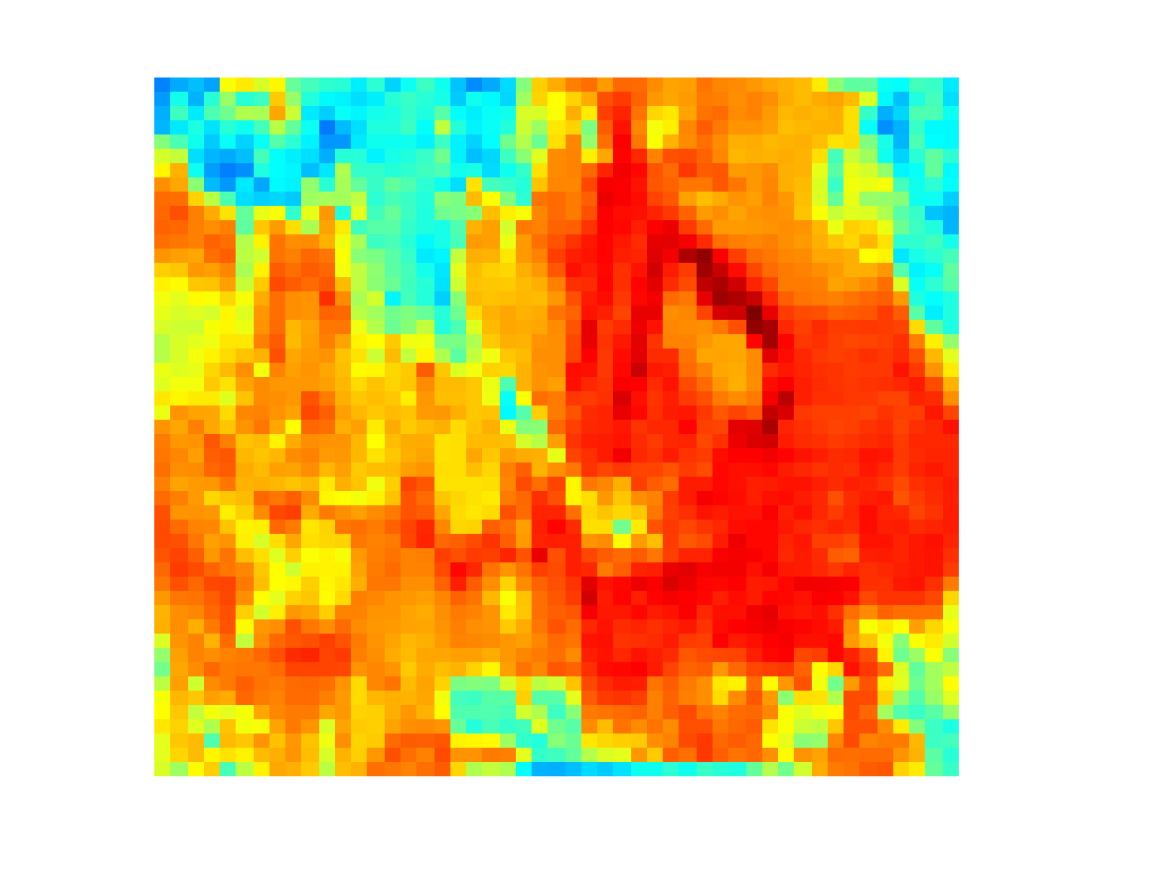}
        \put(0,23) {\rotatebox{90}{Actual}}
        \put(35,70) {hr 1}
        \end{overpic} 
    \end{minipage}\hspace{-0.5mm}
    \begin{minipage}{.195\textwidth}
        \begin{overpic}[width=\linewidth]{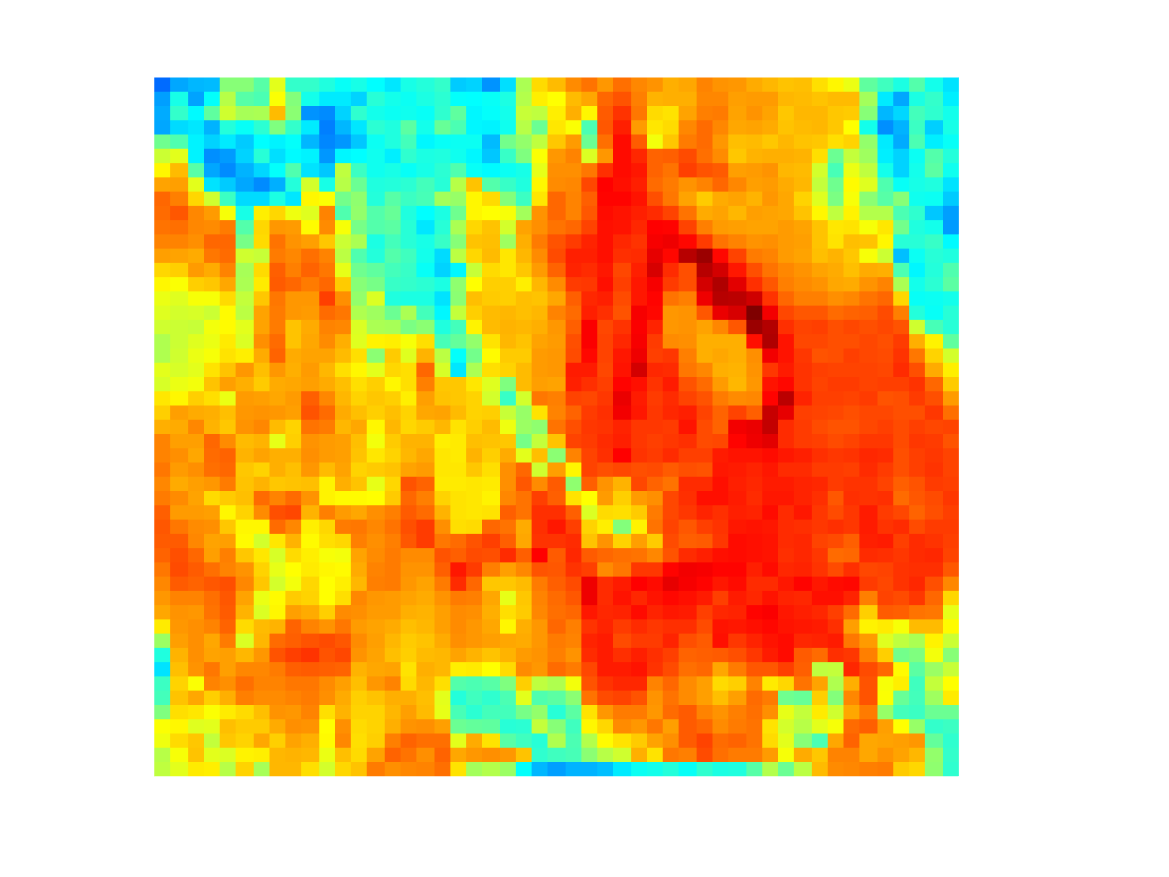}
        \put(35,70) {hr 2}
        \end{overpic}
    \end{minipage}\hspace{-0.5mm}
    \begin{minipage}{.195\textwidth}
        \begin{overpic}[width=\linewidth]{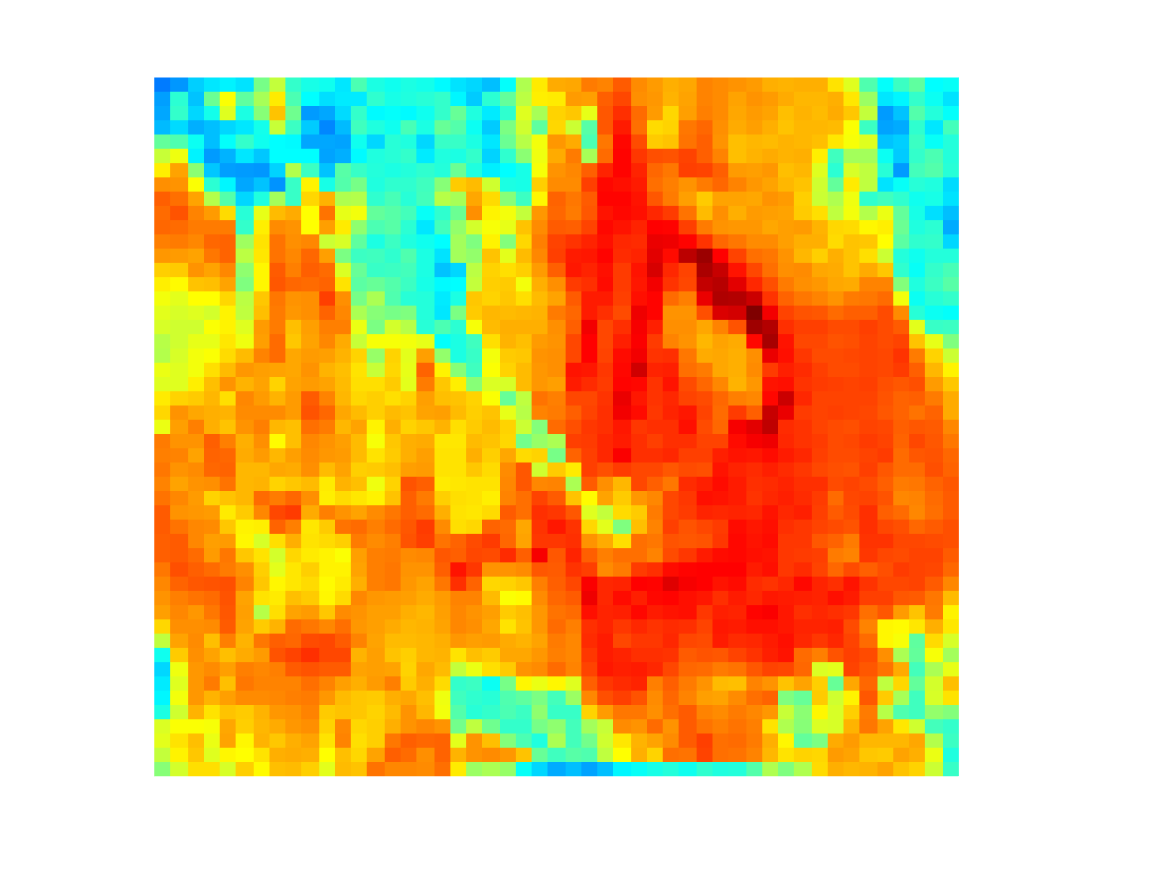}
        \put(35,70) {hr 3}
        \end{overpic}
    \end{minipage}\hspace{-0.5mm}
    \begin{minipage}{.195\textwidth}
        \begin{overpic}[width=\linewidth]{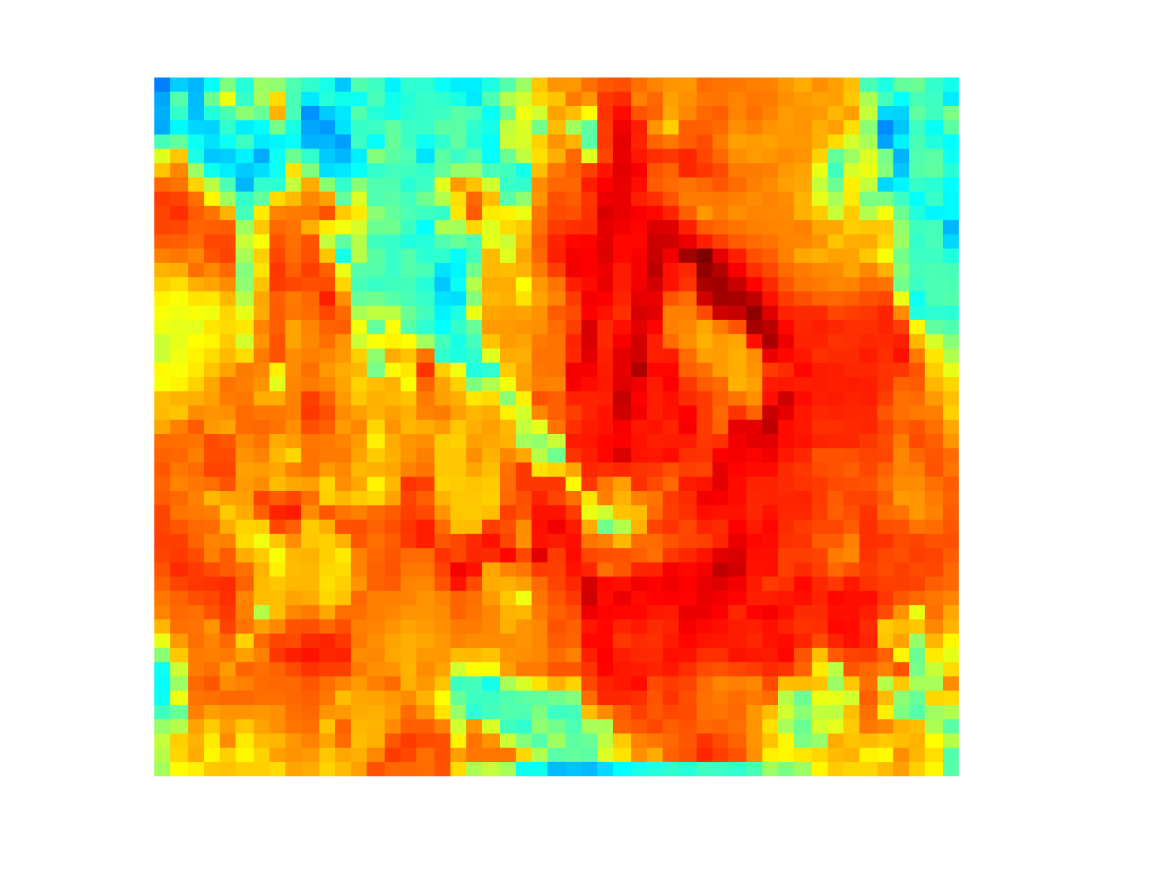}
        \put(35,70) {hr 4}
        \end{overpic}
    \end{minipage}\hspace{-0.5mm}
    \begin{minipage}{.195\textwidth}
        \begin{overpic}[width=\linewidth]{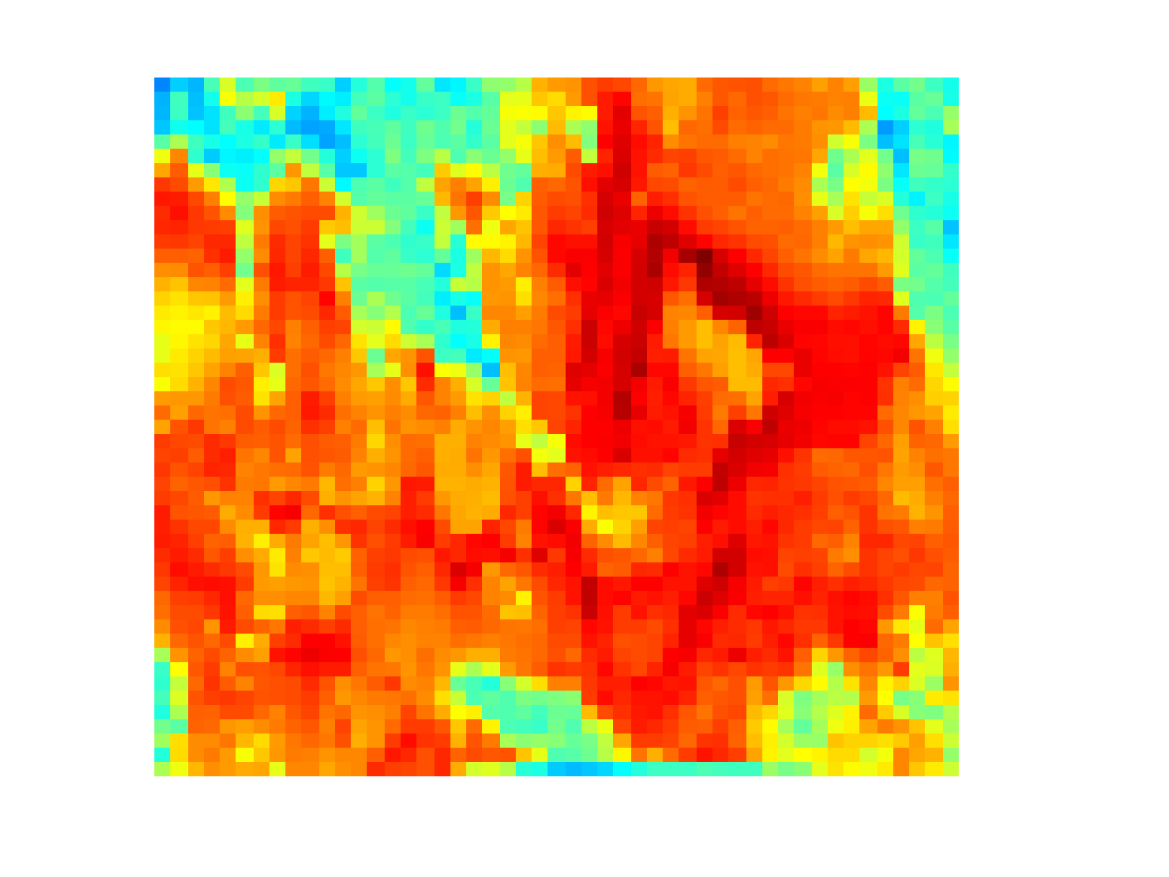}
        \put(35,70) {hr 5}
        \end{overpic}
    \end{minipage}
    
    \vspace{-2mm}

    \begin{minipage}{.195\textwidth}
        \begin{overpic}[width=\linewidth]{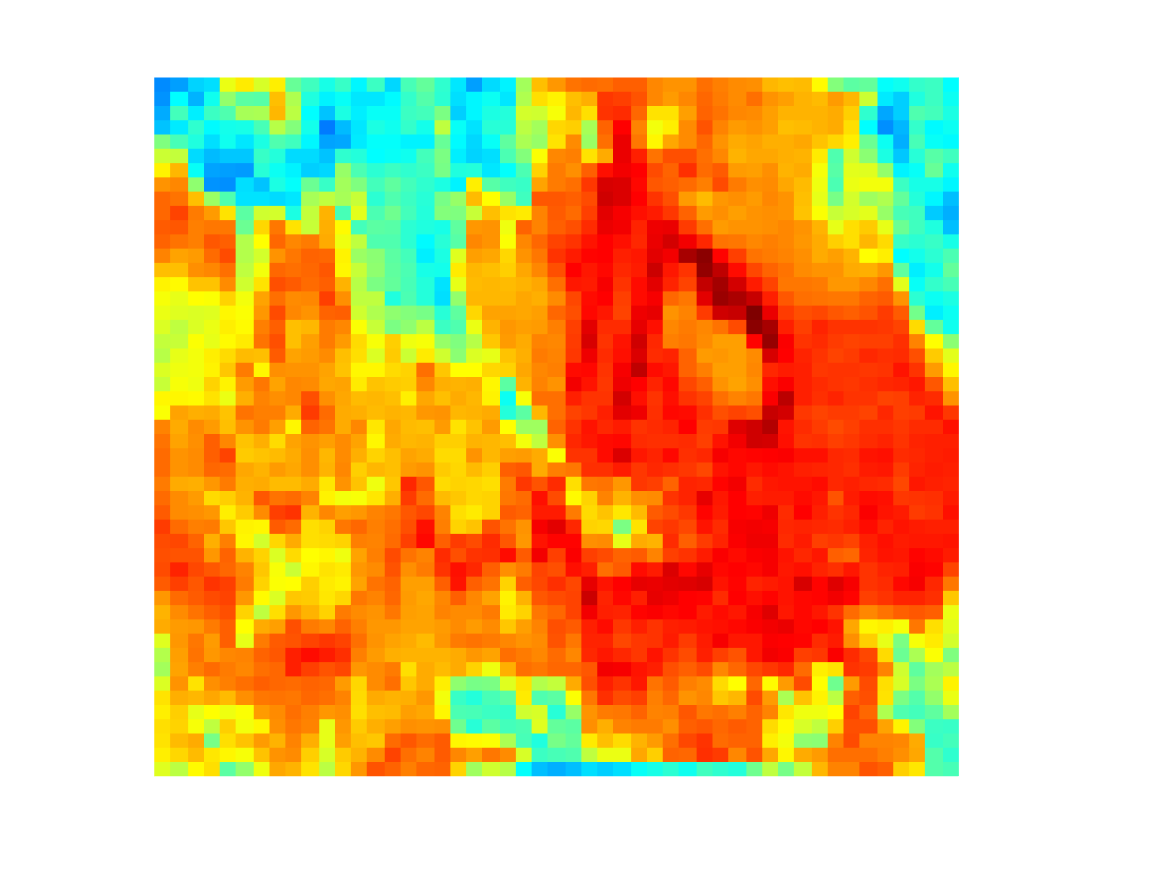}
        \put(0,23) {\rotatebox{90}{DKO}}
        \end{overpic}
    \end{minipage}\hspace{-0.5mm}
    \begin{minipage}{.195\textwidth}
        \includegraphics[width=\linewidth]{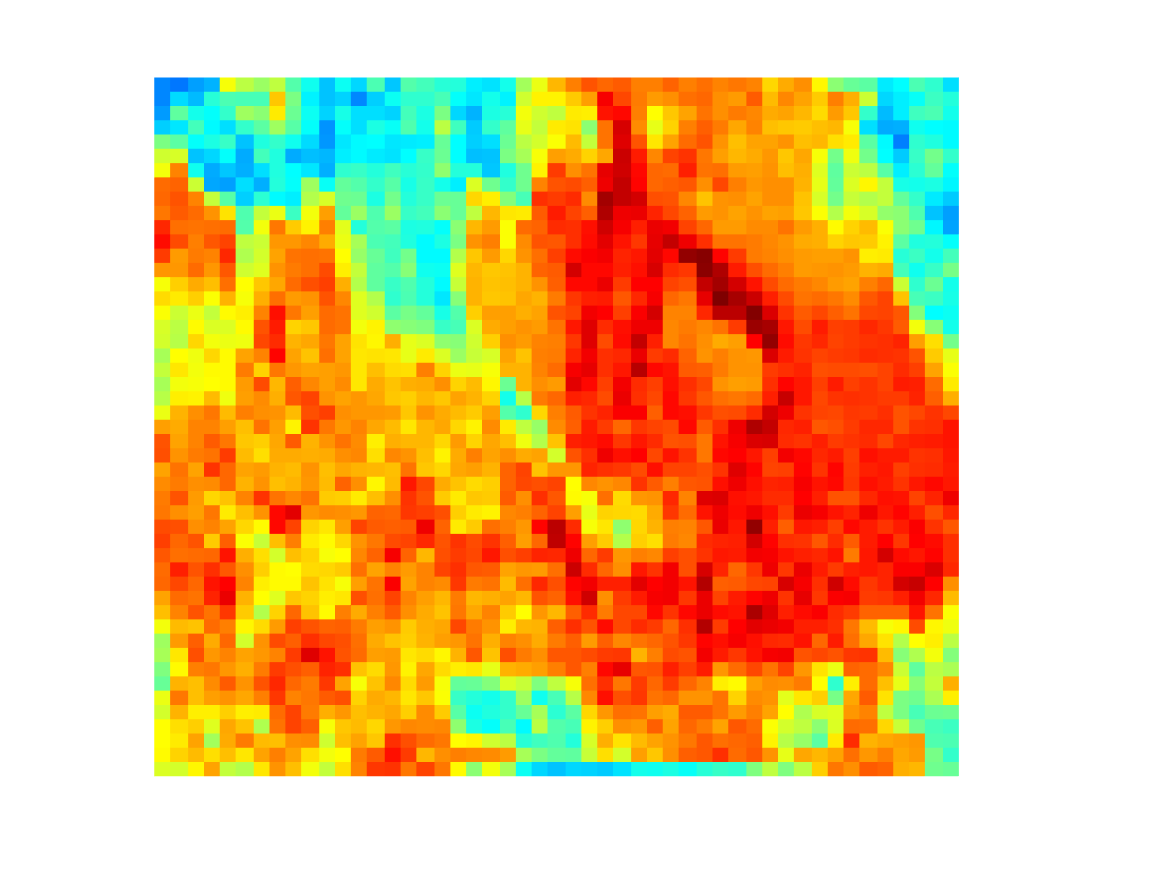}
    \end{minipage}\hspace{-0.5mm}
    \begin{minipage}{.195\textwidth}
        \includegraphics[width=\linewidth]{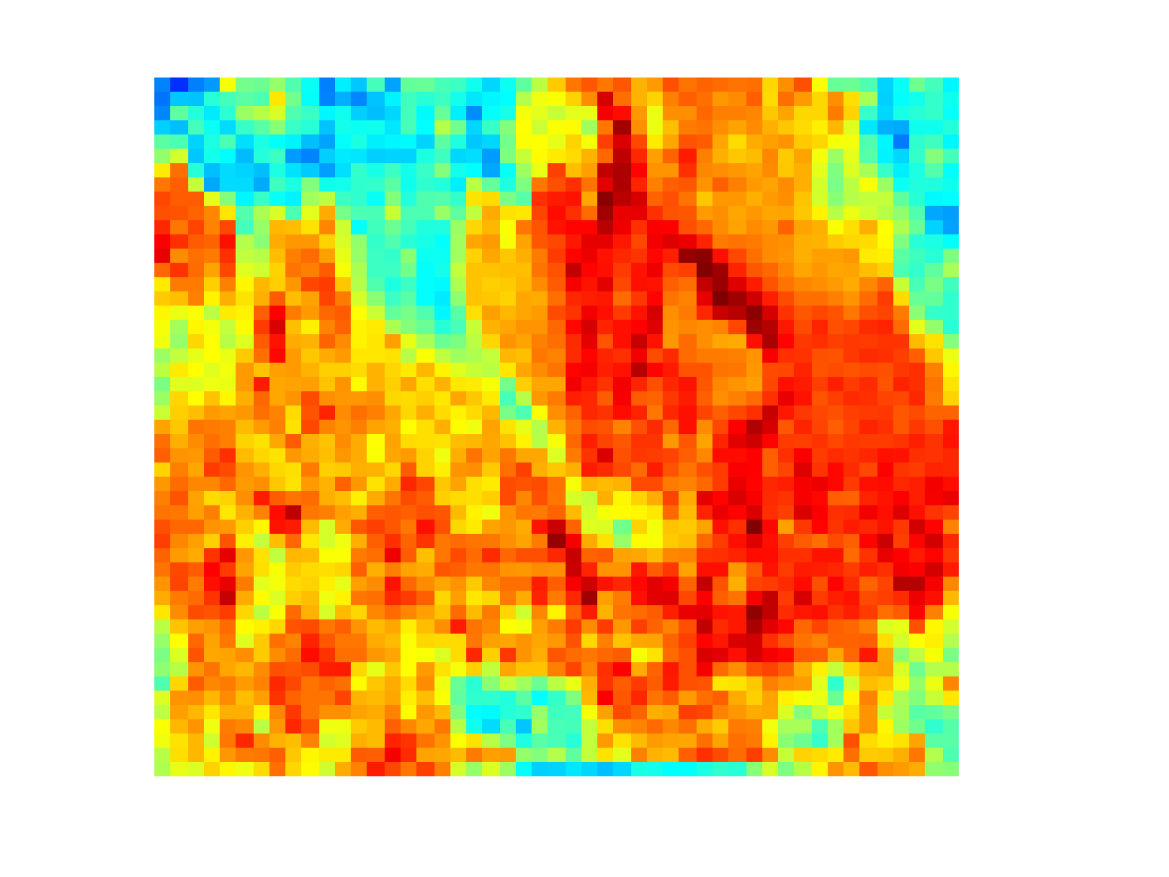}
    \end{minipage}\hspace{-0.5mm}
    \begin{minipage}{.195\textwidth}
        \includegraphics[width=\linewidth]{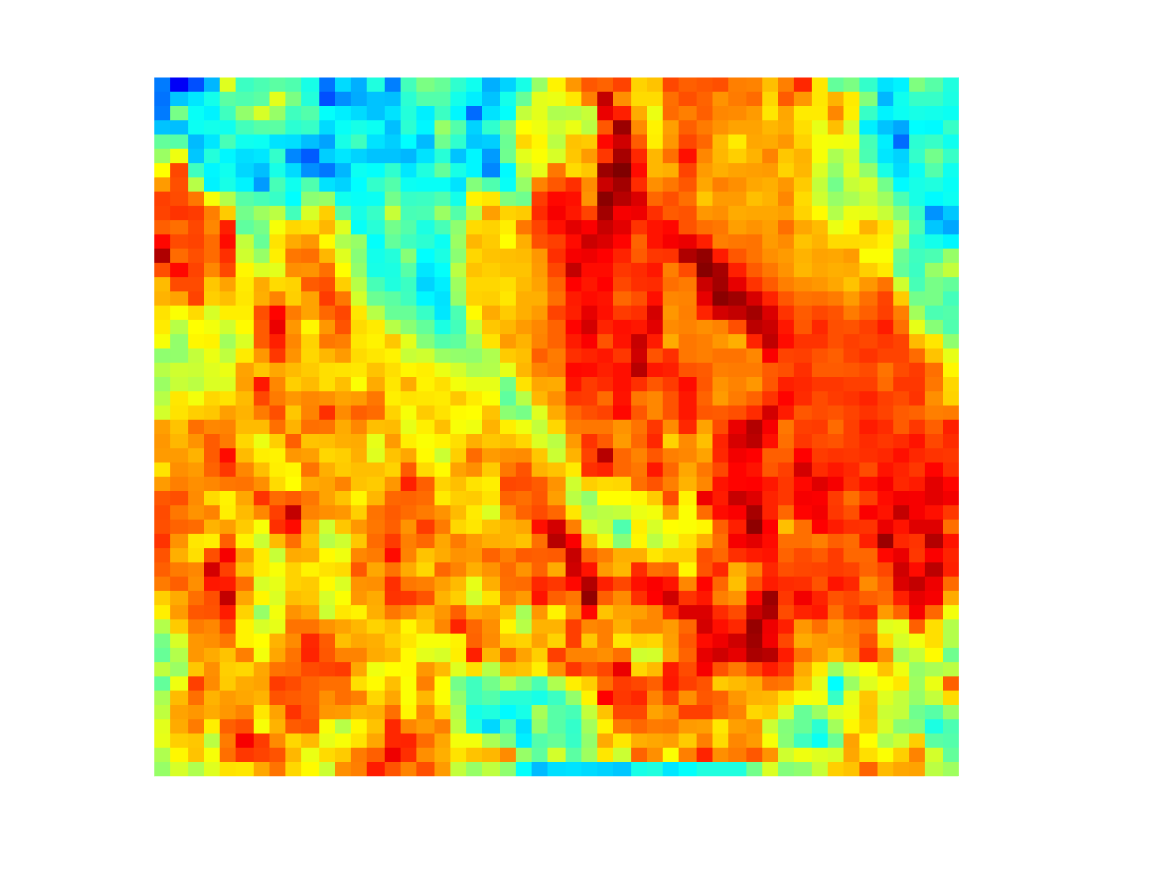}
    \end{minipage}\hspace{-0.5mm}
    \begin{minipage}{.195\textwidth}
        \includegraphics[width=\linewidth]{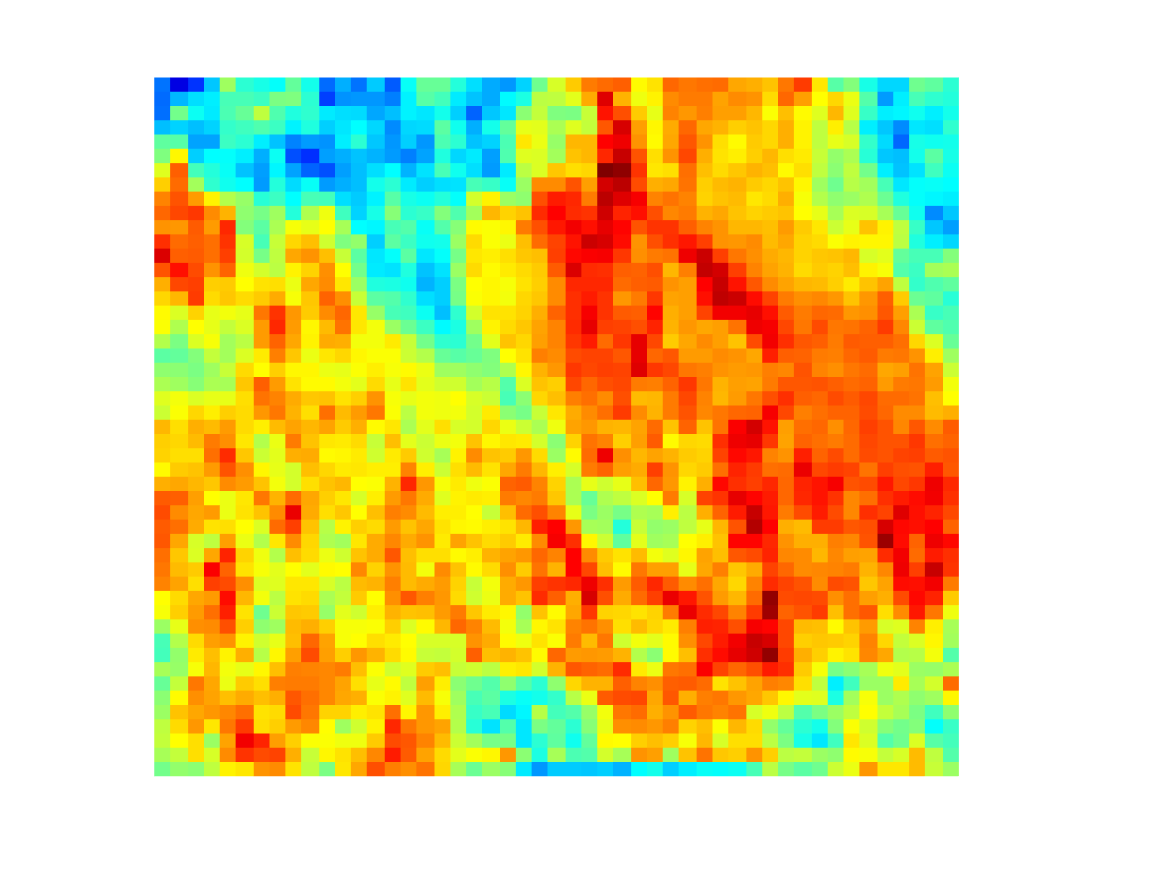}
    \end{minipage}
    \caption{Top row: Observables of the dataset for the first five hours after data stops. Bottom row: Predicted measurement observations using DKO. }
    \label{fig:dust_pred}
\end{figure}

\section{Future Directions}\label{sec:future}
This work lays the foundation for the study of DKOs and several important directions remain to be explored, which we briefly detail.\smallskip

\begin{description}[noitemsep,leftmargin=*]
\item[Learning observables for DKO:]
The performance of the finite-dimensional DKO approximation depends heavily on the choice of observables. While this work focuses on linear and quadratic observables that are analytically prescribed, one could imagine learning the observables in a data-driven way~\cite{li2017extended}. \smallskip

\item[Functional analysis over the probability space:]
To construct a Hilbert space structure for the DKO, we introduce the concept of a random measure serving an analogous role to the Lebesgue measure over finite-dimensional spaces, allowing us to extend the framework of functional analysis to the space of probability distributions and investigate the behavior of operators within this context. This framework has the potential to further expand measure-theoretic techniques in the study of data-driven dynamical systems and related fields.\smallskip

\item[Optimization over measure-valued dynamics:]
Problems in stochastic control, inverse problems, and reinforcement learning involve optimizing over dynamics of distributions rather than trajectories. The DKO provides a natural framework for analyzing such problems.\smallskip


\item[Statistical learning theory for DKO:]
From a theoretical point of view, understanding the sample complexity, generalization guarantees, and concentration inequalities for DKO approximations remains an interesting direction. While convergence in the infinite-data limit is established, rigorous bounds for finite-sample performance would be useful to characterize the reliability of DKO in practice.\smallskip

\item[Applications and real-world systems:]
Finally, applying DKO to real-world systems such as ensemble weather forecasting, flow of population-level dynamics, or epidemiological data could demonstrate the practical utility of the approach. Of particular interest are systems for which only partial, aggregate, or non-tracking measurements are available, where the DKO framework has a clear advantage over trajectory-based methods.
\end{description}

\section*{Acknowledgments} 
M.O.~was supported by the Office of Naval
Research under Grant Number N00014-23-1-2729 over the summer of 2024. M.O.~and Y.Y.~were supported by National Science Foundation through grant DMS-2409855, Office of Naval Research through
grant N00014-24-1-2088, and Cornell PCCW Affinito-Stewart Grant. A.T.~was supported by the Office of Naval
Research under Grant Number N00014-23-1-2729 and National Science Foundation CAREER (DMS-2045646). 


\appendix

\section{The Semigroup Property of the Transfer Operator}\label{app:proof1}
As expected,  the transfer operator for an RDS (see~\cref{def:transfer}) satisfies the semigroup property. This also implies that the DKO has a semigroup structure. 
\begin{lemma}\label{lemma:Pt_semi_group}
The transfer operator $T_t$ satisfies the semigroup property:
\[
T_{t+s} = T_t\circ T_s\qquad t,s\geq 0.
\]
\end{lemma}
\begin{proof}
We start by writing the definition of the probability distribution $T_{t+s}(\pi)$ associated with the transfer operator $T_{t+s}$ given an arbitrary $\pi\in \mathcal{P}(M)$. For any proper test function $h$ defined on $M$, we have
\[
\int_M h(x) \, dT_{t+s}(\pi)(x) = \int_\Theta \int_M h(x) \, d\Bigl(\Phi_{t+s}(\omega,\cdot)\#\pi\Bigr)(x)\,dp(\omega).
\]
Next, by the cocycle property of the evolution mapping $\Phi$, we can decompose 
\[
\Phi_{t+s}(\omega,\cdot) = \Phi_t\bigl(\theta_s\omega,\cdot\bigr) \circ \Phi_s(\omega,\cdot).
\]
Using this semigroup property,  the previous equality becomes
\[
\int_M h(x) \, dT_{t+s}(\pi)(x) = \int_\Theta \int_M h(x) \, d\Bigl(\Phi_t(\theta_s\omega,\cdot)\circ \Phi_s(\omega,\cdot)\#\pi\Bigr)(x)\,dp(\omega).
\]
By the definition of the pushforward measure, the inner integral can be expressed as
\[
\int_M h\Bigl(\Phi_t(\theta_s\omega,x)\Bigr) \, d\Bigl(\Phi_s(\omega,\cdot)\#\pi\Bigr)(x).
\]
Thus, we obtain
\begin{equation}\label{eq:appendix1}
\int_M h(x) \, dT_{t+s}(\pi)(x) = \int_\Theta \int_M h\Bigl(\Phi_t(\theta_s\omega,x)\Bigr) \, d\Bigl(\Phi_s(\omega,\cdot)\#\pi\Bigr)(x)\,dp(\omega).
\end{equation}

Now, we change variables by letting $\tilde{\omega} = \theta_s\omega$. Since the random variables $\theta_s\omega$ and $\omega$ are independent with respect to the probability distribution $p$, we may rewrite the right-hand side as 
\[
\int_\Theta\int_\Theta \int_M h\Bigl(\Phi_t(\tilde\omega,x)\Bigr) \, d\Bigl(\Phi_s(\omega,\cdot)\#\pi\Bigr)(x)\,dp(\tilde{\omega})\,dp(\omega).
\]
We can change the order of integration by Fubini's theorem.
Defining
\[
\tilde{h}(x) =\int_\Theta h\Bigl(\Phi_t(\tilde{\omega},x)\Bigr) \,dp(\tilde{\omega}),
\]
we find that~\cref{eq:appendix1} becomes
\[
\int_M h(x)\, dT_{t+s}(\pi)(x) = \int_\Theta \int_M \tilde{h}(x)\, d\Bigl(\Phi_s(\omega,\cdot)\#\pi\Bigr)(x)\,dp(\omega).
\]
By the definition of the transfer operator $T_s$, we recognize that the inner integral of the right-hand side corresponds to
\[
\int_M \tilde{h}(x)\, d\Bigl(T_s(\pi)\Bigr)(x).
\]
Finally, by applying the definition of $T_t$ to this expression, we obtain
\[
\int_M h(x)\, dT_{t+s}(\pi)(x) = \int_M h(x)\, d\Bigl(T_t\bigl(T_s(\pi)\bigr)\Bigr)(x)\,.
\]
Since this holds for every test function $h$, we conclude that
\[
T_{t+s}(\pi) = T_t\bigl(T_s(\pi)\bigr),
\]
which verifies the semigroup property of the transfer operator for the RDS.
\end{proof}
\end{document}